\documentclass[11pt]{amsart}

\input{preamble}

\newcommand{\Gal}{{\operatorname{Gal}}}

\let\oldeqref\eqref
\makeatletter
\RenewDocumentCommand\eqref{s m}{%
  \IfBooleanTF#1%
  {\textup{\tagform@{\ref*{#2}}}}
  {\oldeqref{#2}}
}
\makeatother

\newcommand{\FE}[1]{{\mathcal {FE}({#1})}}
\newcommand{\SE}[1]{{\mathcal {SE}({#1})}}
\newcommand{\FEbar}[1]{{\overline{\mathcal {FE}({#1})}}}
\newcommand{\FU}[1]{{\mathcal {FU}({#1})}}
\newcommand{\Amp}{{\operatorname{Amp}}}
\newcommand{\bir}{{\operatorname{bir}}}
\newcommand{\Movint}{{\operatorname{Mov}^\circ}}
\newcommand{\BAmp}{{\operatorname{BAmp}}}
\newcommand{\BNef}{{\operatorname{BNef}}}
\newcommand{\Pos}{{\operatorname{Pos}}}
\renewcommand{\Big}{{\operatorname{Big}}}
\newcommand{\Bir}{{\operatorname{Bir}}}

\newcommand{\Mov}{{\operatorname{Mov}}}
\newcommand{\Movbar}{{\ol{\operatorname{Mov}}}}
\newcommand{\Eff}{{\operatorname{Eff}}}

\newcommand{\Exc}{{\operatorname{Exc}}}
\newcommand{\Mth}{{\operatorname{Mon}^{2,\lt}_{\operatorname{Hdg}}}}
\newcommand{\Mtb}{{\operatorname{Mon}^{2,\lt}_{\operatorname{bir}}}}

\newcommand{\Mt}{{\operatorname{Mon}^{2,\lt}}}

\newcommand{\tf}{{\mathrm{tf}}}

\newcommand{\sB}{{\mathcal B}}

\newcommand{\sO}{{\mathcal O}}

\newcommand{\sX}{{\mathcal X}}

\newcommand{\scrX}{{\mathscr X}}

\newcommand{\C}{{\mathbb C}}

\newcommand{\Q}{{\mathbb Q}}
\newcommand{\R}{{\mathbb R}}

\newcommand{\U}{{\mathbb U}}
\newcommand{\V}{{\mathbb V}}

\newcommand{\Z}{{\mathbb Z}}

\newcommand{\Aut}{\operatorname{Aut}}

\newcommand{\codim}{\operatorname{codim}}

\newcommand{\Def}{{\operatorname{Def}}}

\newcommand{\GL}{\operatorname{GL}}

\newcommand{\Hdg}{\operatorname{Hdg}}

\newcommand{\id}{{\rm id}}

\newcommand{\isom}{{\ \cong\ }}

\newcommand{\lt}{{\rm{lt}}}

\newcommand{\Nef}{{\operatorname{Nef}}}

\renewcommand{\O}{{\rm O}}

\newcommand{\ol}[1]{{\overline{#1}}}

\newcommand{\Pic}{\operatorname{Pic}}

\newcommand{\ratl}{\dashrightarrow}

\newcommand{\reg}{{\operatorname{reg}}}

\newcommand{\sing}{{\operatorname{sing}}}

\renewcommand{\to}[1][]{\xrightarrow{\ #1\ }}

\newcommand{\tensor}{\otimes}

\newcommand{\wt}[1]{{\widetilde{#1}}}

\makeatletter
\newcommand*{\da@rightarrow}{\mathchar"0\hexnumber@\symAMSa 4B }
\newcommand*{\da@leftarrow}{\mathchar"0\hexnumber@\symAMSa 4C }
\newcommand*{\xdashrightarrow}[2][]{%
  \mathrel{%
    \mathpalette{\da@xarrow{#1}{#2}{}\da@rightarrow{\,}{}}{}%
  }%
}
\newcommand{\xdashleftarrow}[2][]{%
  \mathrel{%
    \mathpalette{\da@xarrow{#1}{#2}\da@leftarrow{}{}{\,}}{}%
  }%
}
\newcommand*{\da@xarrow}[7]{%
  \sbox0{$\ifx#7\scriptstyle\scriptscriptstyle\else\scriptstyle\fi#5#1#6\m@th$}%
  \sbox2{$\ifx#7\scriptstyle\scriptscriptstyle\else\scriptstyle\fi#5#2#6\m@th$}%
  \sbox4{$#7\dabar@\m@th$}%
  \dimen@=\wd0 %
  \ifdim\wd2 >\dimen@
    \dimen@=\wd2 %
  \fi
  \count@=2 %
  \def\da@bars{\dabar@\dabar@}%
  \@whiledim\count@\wd4<\dimen@\do{%
    \advance\count@\@ne
    \expandafter\def\expandafter\da@bars\expandafter{%
      \da@bars
      \dabar@ 
    }%
  }%
  \mathrel{#3}%
  \mathrel{%
    \mathop{\da@bars}\limits
    \ifx\\#1\\%
    \else
      _{\copy0}%
    \fi
    \ifx\\#2\\%
    \else
      ^{\copy2}%
    \fi
  }%
  \mathrel{#4}%
}
\makeatother

\newtheoremstyle{citing}
  {}
  {}
  {\itshape}
  {}
  {\bfseries}
  {\textbf{.}}
  {.5em}
  {\thmnote{#3}}

\theoremstyle{plain}

\newtheorem{theorem}[subsection]{Theorem}

\newtheorem{lemma}[subsection]{Lemma}
\newtheorem{corollary}[subsection]{Corollary}

\newtheorem{proposition}[subsection]{Proposition}

\theoremstyle{remark}
\newtheorem{example}[subsection]{Example}

\theoremstyle{definition}
\newtheorem{conjecture}[subsection]{Conjecture}
\newtheorem{definition}[subsection]{Definition}

\numberwithin{equation}{section}

\theoremstyle{remark}
\newtheorem{remark}[subsection]{Remark}

{\theoremstyle{citing}
}

\makeatletter
\newsavebox\myboxA
\newsavebox\myboxB
\newlength\mylenA

\newcommand*\xtilde[2][0.8]{%
    \sbox{\myboxA}{$\m@th#2$}%
    \setbox\myboxB\null
    \ht\myboxB=\ht\myboxA%
    \dp\myboxB=\dp\myboxA%
    \wd\myboxB=#1\wd\myboxA
    \sbox\myboxB{$\m@th\widetilde{\copy\myboxB}$}
    \setlength\mylenA{\the\wd\myboxA}
    \addtolength\mylenA{-\the\wd\myboxB}%
    \ifdim\wd\myboxB<\wd\myboxA%
       \rlap{\hskip 0.5\mylenA\usebox\myboxB}{\usebox\myboxA}%
    \else
        \hskip -0.5\mylenA\rlap{\usebox\myboxA}{\hskip 0.5\mylenA\usebox\myboxB}%
    \fi}

\newbox\usefulbox

\def\getslant #1{\strip@pt\fontdimen1 #1}

\def\xxtilde #1{\mathchoice
 {{\setbox\usefulbox=\hbox{$\m@th\displaystyle #1$}%
    \dimen@ \getslant\the\textfont\symletters \ht\usefulbox
    \divide\dimen@ \tw@ 
    \kern\dimen@ 
    \xtilde{\kern-\dimen@ \box\usefulbox\kern\dimen@ }\kern-\dimen@ }}
 {{\setbox\usefulbox=\hbox{$\m@th\textstyle #1$}%
    \dimen@ \getslant\the\textfont\symletters \ht\usefulbox
    \divide\dimen@ \tw@ 
    \kern\dimen@ 
    \xtilde{\kern-\dimen@ \box\usefulbox\kern\dimen@ }\kern-\dimen@ }}
 {{\setbox\usefulbox=\hbox{$\m@th\scriptstyle #1$}%
    \dimen@ \getslant\the\scriptfont\symletters \ht\usefulbox
    \divide\dimen@ \tw@ 
    \kern\dimen@ 
    \xtilde{\kern-\dimen@ \box\usefulbox\kern\dimen@ }\kern-\dimen@ }}
 {{\setbox\usefulbox=\hbox{$\m@th\scriptscriptstyle #1$}%
    \dimen@ \getslant\the\scriptscriptfont\symletters \ht\usefulbox
    \divide\dimen@ \tw@ 
    \kern\dimen@ 
    \xtilde{\kern-\dimen@ \box\usefulbox\kern\dimen@ }\kern-\dimen@ }}%
 {}}

\newcommand*\xoverline[2][0.75]{%
    \sbox{\myboxA}{$\m@th#2$}%
    \setbox\myboxB\null
    \ht\myboxB=\ht\myboxA%
    \dp\myboxB=\dp\myboxA%
    \wd\myboxB=#1\wd\myboxA
    \sbox\myboxB{$\m@th\overline{\copy\myboxB}$}
    \setlength\mylenA{\the\wd\myboxA}
    \addtolength\mylenA{-\the\wd\myboxB}%
    \ifdim\wd\myboxB<\wd\myboxA%
       \rlap{\hskip 0.5\mylenA\usebox\myboxB}{\usebox\myboxA}%
    \else
        \hskip -0.5\mylenA\rlap{\usebox\myboxA}{\hskip 0.5\mylenA\usebox\myboxB}%
    \fi}

\def\xxoverline #1{\mathchoice
 {{\setbox\usefulbox=\hbox{$\m@th\displaystyle #1$}%
    \dimen@ \getslant\the\textfont\symletters \ht\usefulbox
    \divide\dimen@ \tw@ 
    \kern\dimen@ 
    \overline{\kern-\dimen@ \box\usefulbox\kern\dimen@ }\kern-\dimen@ }}
 {{\setbox\usefulbox=\hbox{$\m@th\textstyle #1$}%
    \dimen@ \getslant\the\textfont\symletters \ht\usefulbox
    \divide\dimen@ \tw@ 
    \kern\dimen@ 
    \xoverline{\kern-\dimen@ \box\usefulbox\kern\dimen@ }\kern-\dimen@ }}
 {{\setbox\usefulbox=\hbox{$\m@th\scriptstyle #1$}%
    \dimen@ \getslant\the\scriptfont\symletters \ht\usefulbox
    \divide\dimen@ \tw@ 
    \kern\dimen@ 
    \xoverline{\kern-\dimen@ \box\usefulbox\kern\dimen@ }\kern-\dimen@ }}
 {{\setbox\usefulbox=\hbox{$\m@th\scriptscriptstyle #1$}%
    \dimen@ \getslant\the\scriptscriptfont\symletters \ht\usefulbox
    \divide\dimen@ \tw@ 
    \kern\dimen@ 
    \xoverline{\kern-\dimen@ \box\usefulbox\kern\dimen@ }\kern-\dimen@ }}%
 {}}
\makeatother

{\theoremstyle{definition}

\input{title}

\begin{document}
\thispagestyle{empty}

\begin{abstract}
We prove the Morrison--Kawamata cone conjecture for projective primitive symplectic varieties with $\Q$-factorial and terminal singularities with $b_2\geq 5$, from which we derive for instance the finiteness of  minimal models of such varieties, up to isomorphisms. To prove the conjecture we establish along the way some results on the monodromy group which may be interesting in their own right, such as the fact that reflections in prime exceptional divisors are integral Hodge monodromy operators which, together with monodromy operators provided by birational transformations, yield a semidirect product decomposition of the monodromy group of Hodge isometries.
\end{abstract}

\makeatletter
\@namedef{subjclassname@2020}{
	\textup{2020} Mathematics Subject Classification}
\makeatother


\maketitle

\setlength{\parindent}{1em}
\setcounter{tocdepth}{1}



\tableofcontents

\section{Introduction}\label{section intro}
\thispagestyle{empty}

The celebrated Cone Theorem says that for e.g. a smooth, projective variety $X$ the $K_X$-negative part of the cone of curves of $X$ is
locally rational polyhedral and extremal $K_X$-negative rays yield contraction morphisms (see e.g. \cite[Theorem~3.7]{KM98}). Getting information about the part of the cone where $K_X$ is nonnegative is much more complicated. For instance $K3$ surfaces may have round cone of curves or contain infinitely many $(-2)$-curves (see \cite{totaro-sur}). 
The Morrison--Kawamata cone conjecture provides a beautiful conjectural explanation for those apparent pathologies in terms of the action of the automorphism group. 

Morrison \cite{Mor93,Mor94} gave the original statement of the cone conjecture for Calabi–Yau threefolds. The statement was then generalised by Kawamata in \cite{Kawa97} to families of varieties with numerically trivial canonical bundle, and from there to the so-called klt Calabi–Yau pairs by Totaro, cf. \cite{totaro}. 
We state it here without boundary (cf. \cite[Conjecture 1.12]{Kawa97} and \cite[Conjecture 1.1]{totaro} for pairs).  Recall that a {\it $K$-trivial fiber space} is a  proper surjective morphism $f:X\to S$ with connected fibers between normal varieties such that $X$ has $\Q$-factorial and terminal singularities and $K_X$ is zero in $N^1(X/S)$. 

\begin{conjecture}[The Morrison-Kawamata Cone conjecture]\label{conj:cone}
Let $X\to S$ be a $K$-trivial fiber space.
\begin{enumerate}
\item[(1)] 
There exists a rational polyhedral cone $\Pi$ which is a fundamental domain for the
action of $\Aut(X/S)$ on $\Nef^e(X/S):=\Nef(X/S)\cap \Eff(X/S)$ in the sense that
\begin{enumerate}
\item[(a)] $\Nef^e(X/S)= \cup_{g\in \Aut(X/S)} g^*\Pi$.
\item[(b)] $int (\Pi) \cap int (g^*\Pi) = \emptyset$, unless $g^* = id$ in $\GL(N^1(X/S))$.
\end{enumerate}
\item[(2)] 
There exists a rational polyhedral cone $\Pi'$ which is a fundamental domain in the sense above for the
action of $\Bir(X/S)$ on $\overline{\Mov}^e(X/S):=\overline{\Mov}(X/S)\cap \Eff(X/S)$. 
\end{enumerate}
\end{conjecture}
Item (2) above is also known as the birational Cone conjecture. 

The conjecture was proved for $K3$ surfaces by Sterk \cite{sterk},  checked for the other $K$-trivial smooth surfaces by Yu. Namikawa in \cite{Nami85} (see also \cite{Kawa97}) and more recently extended to two-dimensional klt pairs by Totaro \cite{totaro}.   
The conjecture has been proved for abelian varieties by Prendergast-Smith \cite{PS12}. Note that in this case the birational and the automorphism version coincide. Kawamata proved the cone conjecture for all 3-dimensional $K$-trivial fiber
spaces over a positive-dimensional base \cite{Kawa97}.  For Calabi--Yau 3-folds, there are results by Oguiso \cite{oguiso2001finiteness}, Szendr\"oi \cite{szendroi1999},
Uehara \cite{Ue04} and Wilson \cite{wilson}, but we are far from a complete solution, see \cite{LOP18} for an account and for further references therein.

For irreducible holomorphic symplectic manifolds, one often considers the analogous conjecture where $\Nef^e(X)$ is replaced by the convex hull $\Nef^+(X)$ of $\Nef(X)\cap \Pic(X)_{\Q}$ inside $\Pic(X)_{\R}$ (and similarly for the convex hull $\overline{\Mov}^+(X)$ of $\overline{\Mov}(X)\cap \Pic(X)_{\Q}$). It is easy to see that $\Nef^e(X)\subset \Nef^+(X)$ (see \cite[Remark~1.4]{MY15}), the other inclusion follows from the SYZ-conjecture and thus holds for all known examples. Therefore, up to the validity of the SYZ-conjecture this is just a different formulation. 
The cone conjecture has been solved completely for irreducible holomorphic symplectic manifolds, due to work of Markman \cite{Mar11}, Markman--Yoshioka \cite{MY15} and Amerik--Verbitsky \cite{AV15, AV17, AV20}, following a strategy that can be traced back to Sterk \cite{sterk}. 

In this paper we concentrate on singular holomorphic symplectic varieties, building upon the above works and pursuing, as in the smooth case, Sterk's strategy.

There are multiple reasons to study the cone conjecture for singular holomorphic symplectic varieties. First of all, the latter are together with complex tori and singular Calabi--Yau varieties the building blocks of ``mildly singular’’ varieties with trivial first Chern class (see \cite{GKKP11, DG18, Dru18, GGK19, HP19, BGL20}), which are themselves the conjectural output of the MMP for smooth projective varieties with zero Kodaira dimension. 

Secondly, for almost all choices of Mukai vectors, moduli spaces of sheaves with that fixed Mukai vector on symplectic surfaces are singular and give rise to singular holomorphic symplectic varieties (see \cite{peregorapagnetta} for all the relevant definitions and Theorem~1.19 therein for the result).

Finally, the cone conjecture is naturally related to the problem of finiteness of (marked) minimal models, see e.g. \cite[Corollary 1.5]{MY15}. Recall that a \emph{minimal model} of a variety $X$ is a variety $X'$ with $\Q$-factorial terminal singularities, birational to $X$, and such that the canonical sheaf $\omega_{X'}$ is nef. A \emph{marked minimal model} of $X$ is a minimal model $X'$ together with a birational map $\phi:X \ratl X'$. For both there is an obvious notion of isomorphism.

Finiteness of  minimal models is known to hold for varieties of general type, cf. \cite[Corollary 1.1.5]{BCHM10}, while a variety of non-maximal Kodaira dimension can have infinitely many  minimal models (see \cite[Section~6.8]{Rei83}, \cite[Example~3.8]{Kawa97}). Yet marked minimal models are conjectured to be finite. The MMP is conceived to produce minimal models — these are mostly singular, hence, if adjunction holds, minimal models should be fibered $f:X\to S$ in singular $K$-trivial varieties $X_s$. Therefore, the cone conjecture for possibly singular $K$-trivial fiber spaces $f:X\to S$ together with the abundance conjecture could lead to the finiteness of minimal models up to isomorphisms for arbitrary varieties, as it is the case for projective primitive symplectic varieties, see Corollary~\ref{cor:finite} below.

Our main result verifies the conjecture for projective primitive symplectic varieties  (cf. Definition \ref{definition primitive symplectic}), which form the largest class of singular holomophic symplectic varieties for which we have a theory similar to the one in the smooth case at our disposal.

\begin{theorem}\label{thm:cone}
Let $X$ be a projective primitive symplectic variety with $\Q$-factorial terminal singularities and assume that $b_2(X)\geq 5$.
\begin{enumerate}
    \item The cone conjecture \ref{conj:cone}, item (1),  for the $\Aut(X)$-action on $\Nef^+(X)$ holds for $X$.

    \item The birational cone conjecture \ref{conj:cone}, item (2),  for the $\Bir(X)$-action on $\overline{\Mov}^+(X)$ holds for $X$.
\end{enumerate}

\end{theorem}
As in \cite{totaro}, \cite[Corollary 1.5]{MY15} and \cite[Theorem 1.3]{Ogui14}  we deduce the following consequences. 
\begin{corollary}\label{cor:finite-cont}
Let $X$ be a projective primitive symplectic variety with $\Q$-factorial terminal singularities and assume that $b_2(X)\geq 5$. Then the number of contractions of $X$ is finite, up to automorphisms. In particular the number of Lagrangian fibrations on $X$ is finite, up to automorphisms.
\end{corollary}
\begin{corollary}\label{cor:finite}
Let $X$ be a projective primitive symplectic variety with $\Q$-factorial terminal singularities and $b_2(X) \geq 5$. Then the number of marked minimal models of $X$ is finite.
\end{corollary}
\begin{corollary}\label{cor:ogui}
Let $X$ be a projective primitive symplectic variety with $\Q$-factorial terminal singularities, $\rho(X)=2$ and assume that $b_2(X)\geq 5$. Then:
\begin{enumerate}
    \item Either both boundary rays of $\Nef (X)$ are rational and $\Aut (X)$ is a finite group
or both boundary rays of $\Nef (X)$ are irrational and $\Aut (X)$ is an infinite group.
Moreover, in the second case, ${\Nef} (X) = \overline{\Mov} (X) = \overline{\Pos} (X)$, where ${\Pos} (X)$ is the
positive cone with respect to the Beauville--Bogomolov--Fujiki form $q_X$, and $\Bir (X) =
\Aut (X)$.
 \item Either both boundary rays of $\overline{\Mov}(X)$ are rational and $\Bir (X)$
is a finite group or both boundary rays of $\overline{\Mov}  (X)$ are irrational and $\Bir (X)$ is an
infinite group.
\end{enumerate}
\end{corollary}

The hypothesis $b_2(X)\geq 5$ in the above results is needed in order to apply the Hodge theoretic global Torelli theorem for primitive symplectic varieties (cf. Theorem \ref{thm:hodge-torelli}) and in order to ensure that the monodromy group has finite order inside the group of isometries of $H^2(X)$ (cf. \cite[Theorem~8.2]{BL18}). Of course, the Cone conjecture holds trivially for primitive symplectic varieties with $b_2(X)=3$.\\
One of the keys to prove the conjecture is the study of the monodromy group and in particular the fact that reflections in prime exceptional divisors (i.e.\! prime $\Q$-Cartier divisors having negative square with respect to the Beauville--Bogomolov--Fujiki quadratic form, see Definition~\ref{definition prime exceptional divisor})  
are integral monodromy operators preserving the Hodge structure, a result that seems interesting in its own right. More generally, we summarize below in one single statement our results on the monodromy group. 

\begin{theorem}\label{thm:mon-intro} Let $X$ be a projective primitive symplectic variety with terminal singularities. 
\begin{enumerate}
\item[(1)] (= Theorem \ref{thm:refl-are-mon}).  Let $\pi:X\to \widebar X$ be a birational morphism of primitive symplectic varieties. Then the natural morphism $p:\Def^\lt(X)\to \Def(\widebar X,\pi)$ is
finite Galois cover whose Galois group is isomorphic to
$\Pi_{B\in \mathcal B}W_B$, where $W_B$ is the folded Weil group defined in (\ref{eq weyl group}).
\item[(2)](= Theorem \ref{theorem reflection}).  Let $E$ be a prime exceptional divisor on $X$. Then the reflection 
$$
R_E: H^2(X,\Q)\to H^2(X,\Q), \quad \alpha \mapsto \alpha -2 \frac{q(E,\alpha)}{q(E,E)} E
$$
is integral and yields a monodromy Hodge isometry. 
\item[(3)](= Theorem \ref{thm:semidirect}). Suppose furthermore that $X$ is $\Q$-factorial. The subgroup $W_{\Exc}$ generated by reflections of all prime exceptional divisors in $X$ is a normal subgroup of $\Mth(X)$ and the group $\Mth(X)$ is equal to the semidirect product of the normal subgroup $W_{\Exc}$ and of the subgroup $\Mtb(X)$  of $\Mth(X)$ consisting of all monodromy operators induced by birational maps.
\end{enumerate}
\end{theorem}
Notice that Theorem \ref{thm:mon-intro}, item (1), which is crucial to prove item (2), is false without the terminality assumption, and there are examples, already in dimension 2, where item (2) fails, see Example \ref{example:counterex}.
  
As for the proof of the cone conjecture, we follow the strategy adopted in the smooth case, which relies on several layers of different results proven over the years by several authors, that we try to acknowledge in the proofs. In particular, we follow Markman \cite{Mar11} to show that the birational cone conjecture holds, and then, as in \cite{MY15} and \cite{AV17}, we deduce the cone conjecture from it and from a boundedness result for squares of wall divisors which we establish using  Amerik's and Verbitsky's work on density of orbits \cite{AV20}. 

To overcome the technical difficulties that arise in the singular case we rely upon the recent progress in the moduli theory of singular symplectic varieties by B. Bakker and the first named author \cite{BL18, BL21} as well as on some technical tools developed  in the singular case in \cite{KMPP19, LMP21, BGL20}. 
For Theorem~\ref{thm:mon-intro}, we also substantially rely on ideas and results due to Markman and Namikawa \cite{Mar10galois,Mar11,Nam10,Nam11}.

\subsection*{Acknowledgments.}
We wish to thank Prof. Namikawa for useful exchanges.
GM was supported by PRIN 2020KKWT53 "CuRVi" and is part of INdAM GNSAGA and acknowledges support from it. CL was supported by the DFG through research grant Le 3093/3-2. GP was supported by the ANR project ``Foliage'' (ANR-16-CE40-0008) and by the CNRS International Emerging Actions ``Birational and arithmetic aspects of orbifolds''.  

\subsection*{Notation and Conventions} 
The term variety will denote an integral separated scheme of finite type over $\C$ in the algebraic setting or an irreducible and reduced separated complex space in the complex analytic setting. 

\section{Symplectic Varieties}\label{section symplectic varieties}

We introduce the relevant classes of singular symplectic varieties and recall basic facts about their second cohomology lattice.

\subsection{Primitive Symplectic Varieties} \label{section symplectic}

A \emph{symplectic form} on a smooth variety $U$ is a holomorphic $2$-form $\sigma$ such that $d\sigma=0$ and $\sigma$ defines a symplectic form in the sense of linear algebra on the tangent bundle of $U$. Recall that a \emph{symplectic variety} in the sense of Beauville \cite{Bea00} is a pair $(X,\sigma)$ consisting of a normal (complex analytic) variety $X$ and a holomorphic symplectic form $\sigma \in H^0(X^\reg,\Omega_X^2)$ on the regular part $X^\reg$ such that there is a resolution of singularities $\pi:Y \to X$ for which $\pi^*\sigma$ extends to a holomorphic form on $Y$. Recall that an \emph{irreducible symplectic manifold} is a simply connected compact Kähler manifold such that $H^0(X,\Omega_X^2)=\C\sigma$ for a symplectic form~$\sigma$.

\begin{definition}\label{definition primitive symplectic}
A normal compact Kähler variety $X$ is called \emph{primitive symplectic} if it satisfies $H^1(X,\sO_X)=0$ and $(H^0(X,(\Omega_{X}^2)^{ \vee\vee})=)\ H^0(X^\reg,\Omega_{X}^2)= \C \sigma$, where $(X,\sigma)$ is a symplectic variety.
\end{definition}

An important feature is that the second cohomology group of a primitive symplectic variety carries a pure Hodge structure (of weight two), cf. e.g. \cite[Corollary 3.5]{BL18}.

\begin{remark}\label{remark definition symplectic}\

\begin{enumerate}
	\item For a normal variety $X$ such that $X^\reg$ has a symplectic form $\sigma$, having canonical singularities is in fact equivalent to Beauville's condition \cite{Bea00} that the pullback of $\sigma$ to a resolution of $X$ extends as a regular $2$-form. By \cite{Elk81} and \cite[Corollary~1.7]{KS21} this is also equivalent to $X$ having rational singularities.
	\item Let $X$ be a normal variety having a symplectic form on its regular part. A resolution of singularities $\pi:Y\to X$ such that the pullback of the symplectic form extends to a symplectic form on $Y$ is referred to as a \emph{symplectic resolution}. A resolution is symplectic if and only if it is crepant.
	\item An irreducible symplectic manifold is primitive symplectic (see e.g. \cite[Section 2]{LMP21}). More generally, there is the notion of an irreducible symplectic variety, see \cite{GKKP11}, and those are also primitive symplectic. The converse fails, take e.g. a singular Kummer K3, but a \emph{smooth} primitive symplectic variety is irreducible symplectic by \cite{Sch20def}.
	\item Note that for a bimeromorphic contraction $\pi:Y \to X$ from a primitive symplectic variety to a complex variety $X$, the variety $X$ is primitive symplectic if and only if it is Kähler. 
\end{enumerate}	
\end{remark}

\begin{example}
Let $S$ be a K3 surface and let $v=mw\in H^*(S,\Z)$ be a Mukai vector and let $H$ be $v$-general. Then, by \cite[Theorem~1.19]{peregorapagnetta}  the moduli space of Gieseker $H$-semistable sheaves on $S$ with Mukai vector $v$ is a primitive symplectic variety, if $v \not= (0, mH, 0)$ or $S$ has Picard
rank one. We checked in \cite[Theorem 2.12]{LMP21} that for moduli spaces of objects there is no restriction on $S$ or $v$.
\end{example}
\begin{example}
Let $X$ be a primitive symplectic variety and let $G\subset \Aut(X)$ be a finite group of automorphisms of $X$ which fix the symplectic form. Then $X/G$ is a primitive symplectic variety.
\end{example}

\subsection{The quadratic form}\label{section bbf form}

Recall that the second cohomology of a primitive symplectic variety carries an integral nondegenerate quadratic form $q_X$, the \emph{Beauville--Bogomolov--Fujiki} form (BBF form in what follows). Due to the work of many people, see \cite[\S 5.1]{BL18} and references therein, we know that this form shares many properties with its counterpart for irreducible symplectic manifolds. This is summarized in the following.

\begin{proposition}\label{proposition bbf form}
Let $X$ be a primitive symplectic variety of dimension $2n$ and let $\sigma\in H^{2,0}(X)$ be the class of the symplectic form. Then
\begin{equation}\label{eq definition bbf form}
q_X(\alpha):= \frac{n}{2}\int_X \alpha^2\cdot(\sigma\cdot\widebar{\sigma})^{n-1} + (1-n) \left(\int_X \alpha\cdot\sigma\cdot(\sigma\cdot\widebar{\sigma})^{n-1}\right)\cdot \left( \int_X \alpha\cdot\widebar{\sigma}\cdot(\sigma\cdot\widebar{\sigma})^{n-1}\right)
\end{equation}
defines a non-degenerate quadratic form which, up to scaling by a real number, is defined over $\Z$. Moreover, the associated bilinear form has signature $(3,b_2(X)-3)$, satisfies $q_X \vert_{(H^{2,0}(X)\oplus H^{0,2}(X))_\R} > 0$, and the orthogonal complement to $H^{2,0}(X)\oplus H^{0,2}(X)$ with respect to the quadratic form $q_X$ equals $H^{1,1}(X)$.
\end{proposition}
\begin{proof}
See Section~5 of \cite{BL18}, especially Lemma~5.3, Lemma~5.7 and the references given there.
\end{proof}

\begin{remark}\label{remark torsion cohomology}
As the singularities of a primitive symplectic variety are rational, see Remark \ref{remark definition symplectic}, the pullback $H^2(X,\Z) \to H^2(Y,\Z)$ to a resolution is injective by a well-known fact, see e.g. \cite[Lemma 2.1]{BL21}. In particular, $H^2(X,\Z)$ is torsion free if $X$ admits an irreducible symplectic resolution. Otherwise we will sometimes have to work with the torsion free part $H^2(X,\Z)_\tf$.
\end{remark}
\begin{remark}\label{rmk:BBduality}
A useful consequence of Proposition \ref{proposition bbf form} is that we can switch back and forth between (rational) homology and cohomology in degree $2$. If $X$ is primitive symplectic and $\alpha \in H^2(X,\Q)$, we write $\alpha^\vee\in H_2(X,\Q)$ for the unique class satisfying
\begin{equation}\label{eq definition dual class}
\alpha^\vee\cdot\beta = q_X(\alpha,\beta) \qquad \forall \beta \in H^2(X,\Q).
\end{equation}
This defines an isomorphism $H^2(X,\Q) \to H_2(X,\Q)$, $\alpha \mapsto \alpha^\vee$. For $c \in H_2(X,\Q)$ we will also denote by $c^\vee \in H^2(X,\Q)$ the preimage of $c$ under this isomorphism so that $c^{\vee\vee}  = c$ and $\alpha^{\vee\vee}=\alpha$. Notice that if $c$ is the class of a curve, then $c^\vee$ is a $(1,1)$-class. Note that even if $c \in H_2(X,\Z)$, we will only have $c^\vee \in H^2(X,\Q)$. It is worthwhile to point out that our definition of $\alpha^\vee$ differs from that of Markman \cite[p.~357]{Mar13} by a scalar multiple, which does however not affect any argument.
\end{remark}

\subsection{Deformation theory of primitive symplectic varieties}\label{section deformations}

A lot is known about deformations of singular symplectic varieties thanks to the pioneering work of Namikawa in \cite{Nam01a,Nam06,Nam08}. Let $X$ be a primitive symplectic variety. Then mainly due to work of Grauert and Douady \cite{Gra74, Dou74}, there is a universal deformation $\scrX \to \Def(X)$ of $X$ over a complex space germ and we refer to $\Def(X)$ as the Kuranishi space of $X$, see e.g. the beginning of \S~4  of \cite{BL18} for general background on Kuranishi spaces and further references. 

For a projective symplectic variety (in the sense of Beauville), Namikawa proved in \cite[Theorem~1]{Nam06} that the Kuranishi space $\Def(X)$ is smooth. Note that Namikawa's unobstructedness result is under the additional hypothesis that a $\Q$-factorial terminalization exists, but this is always fulfilled by \cite[Corollary~1.4.3]{BCHM10}. Moreover, by Corollary~2 and Proposition~1 of \cite{Nam06} if $X$ admits a symplectic resolution, the universal deformation is shown to be a smoothing. Thus, looking at all deformations of $X$ does not give us information about the singularities. Instead it has proven fruitful to consider \emph{locally trivial} deformations only. We refer the reader to \cite[\S~2]{BGL20} for a general reference for locally trivial deformations. 

By \cite[(0.3) Corollary]{FK87} there exists a closed complex subspace $\Def^\lt(X)\subset \Def(X)$ of the Kuranishi space parametrizing the locally trivial deformations of $X$. Its tangent space is given by $H^1(X,T_X)$ and by \cite[Theorem~4.7]{BL18} it is smooth of dimension $h^{1,1}(X)$. We will write 
\begin{equation}\label{eq locally trivial universal deformation}
\scrX \to \Def^\lt(X)
\end{equation}
for the universal locally trivial family of deformations, that is, the restriction of the Kuranishi family to $\Def^\lt(X)$.

\subsection{Deformations of contractions of primitive symplectic varieties}\label{section deformations of contractions}

Let $c:X \to \widebar{X}$ be a proper bimeromorphic morphism between primitive symplectic varieties. Let us recall that by \cite[Lemma~5.21]{BL18} there is a $q_X$-orthogonal decomposition 
\begin{equation}\label{eq:H2}
    H^2(X,\Q)=c^* H^2(\widebar X,\Q)\oplus N
\end{equation}
where the summands are defined over $\Z$ and $N$ is negative definite and of type $(1,1)$. Recall that by \cite[Proposition 11.4]{KM92}, there is a commutative diagram
\begin{equation}
\label{eq defo morphism}
\xymatrix{
\scrX \ar[d]\ar[r]& {\widebar{\mathscr{X}}} \ar[d]\\
\Def(X) \ar[r]^{p} & \Def(\widebar{X}) \\
}
\end{equation}
between the universal deformations of $X$ and $\widebar X$. We will restrict \eqref{eq defo morphism} to the subspace $\Def^\lt(X) \subset \Def(X)$. Then we have  the following slight generalization of \cite[Proposition~5.22]{BL18} (in comparison to which we drop the assumptions on projectivity and $\Q$-factoriality), which is an easy consequence of  \cite[Corollary~2.29]{BGL20}. It says that the locally trivial deformations of $\widebar X$ are identified via $p$ with the locus of locally trivial deformations of $X$ where the classes of contracted curves remain Hodge.

\begin{proposition}\label{prop defo}
Let $X \to \widebar{X}$ be a proper bimeromorphic map between primitive symplectic varieties and let $N$ be as in \eqref{eq:H2}. Denote by $\Def^\lt(X,N)\subset \Def^\lt(X)$ the sublocus of deformations such that all classes in $N$ remain of type $(1,1)$. Then the following holds:
\begin{enumerate}
	\item $p^{-1}(\Def^\lt(\widebar X)) \cap \Def^\lt(X) = \Def^\lt(X,N)\subset \Def^\lt(X)$.
	\item The restriction $p:\Def^\lt(X,N)\to\Def^\lt(\widebar X)$ is an isomorphism.
\end{enumerate}
\end{proposition}

\begin{proof}
We do not know whether $\Def(X) \to \Def(\widebar X)$ is surjective if $X$ is nonprojective, but we only need that $\Def^\lt(\widebar X)$ is contained in the image. Thanks to \cite[Corollary~2.29]{BGL20}, it is actually contained in the image of $\Def^\lt(X,N)$. The rest of the proof is basically the same as for  \cite[Proposition~5.22]{BL18}.
\end{proof}





\begin{proposition}\label{proposition contracted curves deform}
Let $X$ be a primitive symplectic variety and let $c\,:\,X\to \widebar{X}$ be a bimeromorphic contraction to a primitive symplectic variety $\widebar{X}$. Let $R \subset X$ be a contracted curve and let $N \subset H^2(X,\Q)$ be as in (\ref{eq:H2}). Then $R$ deforms over all of $\Def^{\lt}(X,N)$.
\end{proposition}
\begin{proof}
By Proposition~\ref{prop defo}, $\Def^\lt(X,N)$ is isomorphic to $\Def^\lt(\widebar X)$. In particular, diagram \eqref{eq defo morphism} gives rise via restriction to a morphism $\scrX \to \widebar \scrX$ over $\Def^\lt(\widebar X)$ which is a deformation of $c$. By \cite[Corollary~2.29]{BGL20}, even the map $\scrX \to\widebar \scrX$ is  locally trivial\footnote{Clearly, the spaces $\sX, \widebar \sX$ are locally trivial over $\Def^\lt(\widebar X)$. But also the morphism $\sX \to \widebar \sX$ is locally trivial over $\Def^\lt(\widebar X)$ in the sense of \cite[Definition~2.25]{BGL20}. This follows from Proposition~2.26 of op. cit. We will however not need this in what follows.} as a map.  Hence, the curve $R$ deforms. 
\end{proof}

It is worthwhile pointing out that from the local triviality of $\scrX$ resp. $\widebar \scrX$ over $\Def^\lt(\widebar X)$  it is clear that the cycle of $R$ deforms over $\Def^\lt(\widebar X)$. The conclusion of the proposition is however stronger in the relative Picard rank one case.  

\begin{corollary}\label{corollary curve deforms}
Let $X$ be a primitive symplectic variety and let $c\,:\,X\to \widebar{X}$ be a bimeromorphic contraction of relative Picard rank one to a primitive symplectic variety $\widebar{X}$. Let $R \subset X$ be a contracted curve. Then $R$ deforms over all of its Hodge locus.
\end{corollary}
\begin{proof}
As $c$ has Picard rank one, it follows from Proposition~\ref{prop defo} that the Hodge locus of $R$ in $\Def^\lt(X)$ is isomorphic to $\Def^\lt(\widebar X)$. In fact, the lattice $N$ from \eqref{eq:H2} is spanned by the dual $[R]^\vee \in H^2(X,\Q)$ of the class of $R$. Then the claim follows from Proposition~\ref{proposition contracted curves deform}.
\end{proof}

\subsection{Locally trivial monodromy operators}\label{section monodromy}

Let $f: \scrX \to S$ be a locally trivial family of primitive symplectic varieties. By \cite[Lemma~2.4]{BL21}, the torsion free part of the second cohomology of the fibers of $f$ forms a local system so that we have a notion of parallel transport. This is sufficient for our purposes, but it is convenient to use the following strengthening. By \cite[Proposition~5.1]{AV21}, a locally trivial deformation is even topologically trivial. Therefore, if $\gamma$ is a path in $S$ connecting two points $s_1, s_2 \in S$, we consider the parallel transport 
\[
p_\gamma:H^2(X_{s_1},\Z) \to H^2(X_{s_2},\Z)
\]
where $X_{s_i}=f^{-1}(s_i),\ i=1,2$, and refer to it as a \emph{locally trivial parallel transport operator}. Such operators preserve the Bogomolov--Beauville--Fujiki form. 

\begin{definition}\label{definition monodromy group}
Let $X$ be a primitive symplectic variety. By varying $\gamma$ over all loops in all possible locally trivial families, the corresponding parallel transport operators $p_\gamma$ on $H^2(X,\Z)$ form the \emph{locally trivial monodromy group}. We denote it by $\Mt(X) \subset \O(H^2(X,\Z),q_X)$. Note that the monodromy group is contained in the orthogonal group $\O(H^2(X,\Z),q_X)$ as parallel transport preserves the BBF form. We denote by $\Mth(X) \subset \Mt(X)$ the group of \emph{monodromy Hodge isometries}, that is, the subgroup of all those isometries that preserve the Hodge structure. 
\end{definition}

If $b_2(X)\neq 4$, the monodromy group has finite index in the orthogonal group of $H^2(X,\Z)_\tf$ by \cite[Theorem~8.2]{BL18}.

\subsection{Prime Exceptional Divisors} \label{section prime exceptional}

The bimeromorphic geometry of a primitive symplectic variety will be in large parts determined by the geometry of prime exceptional divisors. In the following sections we will try to transport as much as possible from what is known for irreducible symplectic manifolds to primitive symplectic varieties. Here we simply give the definition. 

\begin{definition}\label{definition prime exceptional divisor}
Let $X$ be a primitive symplectic variety. A prime Weil divisor $E \subset X$ is called a \emph{prime exceptional divisor} if it is $\Q$-Cartier and $q_X(E) < 0$.
\end{definition}
In particular, a prime exceptional divisor $E$ has a cohomology class $[E]\in H^2(X,\Q)$.
Recall that by \cite[Theorem 1.2, item (1)]{LMP21} (which is essentially contained in \cite[Theorem A and its proof]{BBP13} and \cite[Theorem 3.3]{Dru11}) a prime exceptional divisor on a projective $\Q$-factorial primitive symplectic variety can be contracted on a birational model (which is a locally trivial deformation of the initial variety). 
This explains why a prime exceptional divisor on a projective primitive symplectic variety is uniruled.

\section{Reflections in prime exceptional divisors}\label{section reflections}

Let $X$ be a projective primitive symplectic variety. It is natural to ask whether reflections in prime exceptional divisors $E \subset X$ are locally trivial monodromy operators. This has been shown by Markman \cite{Mar10galois} for irreducible symplectic manifolds. To obtain the result, we blend his approach with the one adopted by Namikawa in \cite{Nam10} which seems technically less involved for singular varieties than certain parts of Markman's argument. We will prove this under the additional hypothesis that $X$ has terminal singularities. 
Concerning the singularities, this result is actually optimal. Already for K3 surfaces one obtains a counterexample (see  Example~\ref{example:counterex} for a more detailed discussion). 

A key ingredient in Markman's argument is the birational contractibility of prime exceptional divisors on projective irreducible symplectic manifolds which in our case is replaced by \cite[Theorem 1.2]{LMP21}. Since the property of being a monodromy operator is of course invariant under (locally trivial) deformations we will assume in this section that given a prime exceptional divisor $E\subset X$, we have a birational contraction $X\to \widebar X$ of $E$. We want to show that the induced map $\Def^\lt(X) \to \Def(\widebar X)$ from the space of locally trivial deformations of $X$ is a Galois cover onto its image. Note that by diagram (\ref{eq defo morphism}) any small deformation of $X$ induces one of $\widebar X$ but the latter need not be locally trivial even if the first is.

\subsection{Folded Dynkin Diagrams}\label{section folded dynkin}

We consider the following set-up. Let $\pi:X \to \widebar X$ be a bimeromorphic morphism between primitive symplectic varieties and suppose that $X$ has terminal singularities. We denote by $\Sigma\subset \widebar X$ the singular locus. According to Kaledin \cite{Kal06a} (see also \cite[Corollary~1]{Nam01} and \cite[Theorem~3.4]{BL18}) the locus $\Sigma$ is stratified into symplectic varieties. Maximal strata of codimension $2$ parametrize ADE singularities. More precisely, there is a closed subvariety $\Sigma_0\subset \Sigma$ of $\codim_{\widebar X} \Sigma_0\geq 4$
such that for each $p\in \Sigma\setminus \Sigma_0$ there is an isomorphism of analytic germs
\begin{equation}\label{eq kaledin germs}
(\widebar X,p) \isom (\widebar S,0) \times \mathbb (\C^{2n-2},0)
\end{equation}
where $(\widebar S,0)$ is an ADE surface singularity. 
The subvariety $\Sigma_0\subset \widebar X$ is  called the {\it dissident locus}. Let $\mathcal B$ be the set of connected components of 
$\Sigma\setminus \Sigma_0$ and for each connected component $B$ we denote by $W_B$ the Weyl group of the folded root system (cf. \cite[Section 3]{Mar10galois} and \cite[beginning of Section 1]{Nam10}). Let us recall that each $B\in \sB$ determines a root system $\Phi_B$ (the one of $\widebar S$) and a graph automorphism $\tau$ of the Dynkin diagram of $\Phi_B$ (coming from the monodromy of fiber components of $X\to \widebar{X}$ over $B$). We recall that
\begin{equation}\label{eq weyl group}
    W_B=\{ w \in W(\Phi_B) \mid \tau w \tau^{-1}= w\}
\end{equation}
where $W(\Phi_B)$ is the Weyl group associated to $\Phi_B$.

\subsection{Covering between Kuranishi spaces}

Let $X$ be a primitive symplectic variety and $\pi:X \to \widebar X$ a bimeromorphic contraction. Recall from Section \ref{section deformations of contractions} that deformations of $X$ induce deformations of $\widebar X$, more precisely, we have the commutative diagram \eqref{eq defo morphism}. We denote by $\Def(\widebar X,\pi) \subset \Def(\widebar X)$ the image of $\Def^\lt(X)$ under $p$. Let us consider the restriction 
\begin{equation}\label{eq deflt covering contraction}
p:\Def^\lt(X)\to \Def(\widebar{X},\pi).
\end{equation}
 The proof of the following result is almost literally as in \cite[Proposition~5.24]{BL21}, we include it for convenience.

\begin{lemma}\label{lemma deflt finite}
Let $\pi:X \to \widebar X$ be a bimeromorphic map between primitive symplectic varieties. Then the restriction $p:\Def^\lt(X) \to \Def(\widebar X)$ is finite.
\end{lemma}
\begin{proof}
If $X$ is projective, then $p$ is finite by Namikawa's work \cite[Theorem~1]{Nam06}. Suppose that $\widebar X$ is non-projective and $p$ is not finite. Then there is a one-parameter family $\sX \to \widebar X\times \Delta$ over a disk $\Delta$ such that $\sX \to \Delta$ is locally trivial, the classifying map $\Delta \to \Def^\lt(X)$ is not constant and the induced morphism $X_t\to \widebar{X}$ is birational for all $t\in \Delta$.  By \cite[Corollary~6.10]{BL18}, there is $t\in \Delta$ such that $\sX_t$ is projective. As $\pi:\sX_t\to \widebar X$ is birational,  $\widebar X$ is Moishezon and projective by \cite[Proposition~5]{Nam02}, contradiction.
\end{proof}

\begin{lemma}\label{lemma generic fiber isomorphic}
In the situation of Section \ref{section deformations of contractions}, let $t\in \Def^\lt(X)$ be a very general point. Then the corresponding contraction $\pi_t:\scrX_t \to \widebar \scrX_{p(t)}$ is an isomorphism. 
\end{lemma}
\begin{proof}
If $\pi_t$ is an isomorphism there is nothing to prove. Let us assume that the exceptional locus of $\pi_t$ is non-empty. By the local Torelli theorem \cite[Proposition~5.5]{BL18}, the variety $\scrX_t$ has Picard number zero. Using the BBF form as in Remark \ref{rmk:BBduality}, we see that $\scrX_t$ cannot contain any curves. The exceptional locus of $\pi_t:\scrX_t\to\widebar \scrX_t$ is however covered by curves if it is non-empty. Thus, the latter must hold and $\pi_t$ is an isomorphism.
\end{proof}

Recall from \cite[Theorem~4.7]{BL18} that the subspace $\Def^\lt(X) \subset \Def(X)$ is smooth of dimension $h^{1,1}(X)$. In particular, $p:\Def^\lt(X)\to \Def(\widebar X,\pi)$ factors through the normalization $D$ of $\Def(\widebar X,\pi)$.

\begin{theorem}\label{theorem galois cover without group}
Let $\pi:X \to \widebar X$ be a bimeromorphic morphism between primitive symplectic varieties. Then the induced morphism $q:\Def^\lt(X)\to D$ onto the normalization $D \to \Def(\widebar X,\pi)$ is a finite Galois cover. Moreover, the Galois group acts on $H^2(X,\Z)$ via locally trivial Hodge monodromy operators. The action is trivial on  $H^{2,0}(X)$ and faithful on $H^{1,1}(X)$.
\end{theorem}

We would like to point out that in this theorem there is no hypothesis on $X$ being terminal or $\Q$-factorial.

\begin{proof}
Let $U\subset \Def(\widebar X,\pi)$ be the biggest open subset such that $\pi_t:\scrX_t \to\widebar \scrX_{p(t)}$ is an isomorphism for every $t\in V:=p^{-1}(U) \subset \Def^\lt(X)$, and $p:V \to U$ is unramified over $U$. Then $U$ is non-empty by Lemma~\ref{lemma generic fiber isomorphic}. Let us denote by $f:\scrX \to \Def^\lt(X)$ and $g:\widebar \scrX \to \Def(\widebar X,\pi)$ the restrictions of the universal families. By \cite[Lemma~2.4]{BL21} respectively \cite[Proposition~5.1]{AV21}, the sheaf $\V:=\left(R^2f_*\Z_{\scrX}\right)$ is a local system on $\Def^\lt(X)$.  Further, the constructible sheaf $\U:=\left(R^2g_*\Z_{\widebar \scrX}\right)$ becomes a local system after restricting it to $U$ by definition of the latter. Moreover, by unobstructedness of $\Def^\lt(X)$, the local system $\V$ is trivial and hence the canonical isomorphism 
\[
p^{-1}\U\vert_U \to \V\vert_V
\]
shows that $p^{-1}\U$ is a trivial local system over $U$. In particular, if $K\subset \pi_1(U,t)$ for some point $t\in U$ is the kernel of the monodromy representation $\pi_1(U,t) \to \O(H^2(\widebar \scrX_t,\Z))$ and $U_K\to U$ is the associated topological Galois cover, then we have a factorization $V\to U_K \to U$. Arguing via the (locally trivial) period map of $X$ as in \cite[p.~193, proof of Lemma 1.2]{Mar10galois} and using the ``singular'' local Torelli \cite[Proposition~5.5]{BL18} instead of the classical one, we deduce that $V \to U_K$ is an isomorphism and thus $V\to U$ is Galois. The Galois action can be extended to $\Def^\lt( X)$ again via the period map as on p.~194 of \emph{op. cit.} by using the ``singular'' local Torelli instead of the classical one.  If $G=\Gal(p)$ is the Galois group, we consider the diagram

\[
\xymatrix{
\Def^\lt(X) \ar[r]^q \ar@/^20px/[rr]^p & \Def^\lt(X)/G \ar[r]^{\exists\, r} & \Def(\widebar X,\pi)\\
V \ar[r]  \ar@{^(->}[u]& V/G \ar@{=}[r]  \ar@{^(->}[u] & U \ar@{^(->}[u]  \\
}
\]
First note that $p$ factors through $q$ as the latter is a quotient by $G$ and the former is $G$-invariant. We read off that the morphism $r$ is bimeromorphic, and as both $p$ and $q$ are finite (we invoke Lemma~\ref{lemma deflt finite} for $p$ here) and $\Def^\lt(X)/G$ is normal,  $r:\Def^\lt(X)/G \to \Def(\widebar X,\pi)$ is the  normalization. Also the rest of the claim is rather formal and we refer to Markman's paper. For the triviality on $H^{2,0}$, we use the Lemma~\ref{lemma monodromy fixes cohomology downstairs} below. 
\end{proof}

Note that every complex space is locally contractible, so unobstructedness of locally trivial deformations of $X$ was not needed to show that the local system $\V$ is trivial. We used it however to deduce a factorization of $p$ through the normalization of $\Def(\widebar X,\pi)$. Let us have a closer look at the elements of the Galois group.

\begin{lemma}\label{lemma monodromy fixes cohomology downstairs}
The elements of the Galois group $G$ of the cover $p$ act on $H^2(X,\Z)$ fixing the subspace $H^2(\widebar{X},\Z)$. 
\end{lemma}
\begin{proof}
As in the proof of Theorem~\ref{theorem galois cover without group}, we consider the open set $U\subset \Def(\widebar X,\pi)$ over which $\scrX_t \to \widebar \scrX_{p(t)}$ is an isomorphism and recall that the restriction of $\U=\left(R^2g_*\Z_{\widebar \scrX}\right)$ to this open set is a local system. Let us consider the diagram
\[
\xymatrix{
X\ar[r] \ar[d] & \scrX \ar[d]\\
{\widebar{X}} \ar[r] & {\widebar{\scrX}}\\
}
\]
and the induced pullback morphisms on the second cohomology. By Lojasiewicz's theorem \cite{Loj64}, see also \cite[Theorem I.8.8]{BHPV04}, the morphism $H^2(\widebar{\scrX},\Z)\to H^2(\widebar X,\Z)$ is an isomorphism, hence classes in $H^2(\widebar X,\Z)$ extend to sections in  $\U\vert_U$ and thus their restriction to a general fiber is monodromy invariant. The monodromy of the local system acts on $H^2(X,\Z)$ by the triviality of the pullback $p^{-1}\U_U$. By the discussion above, we see that $H^2(\widebar X,\Z)$ is invariant under this action.
\end{proof}

The proof of the following result essentially proceeds as in \cite[Paragraph (1.3)]{Nam10}.

\begin{theorem}\label{thm:refl-are-mon}
Let $\pi:X \to \widebar X$ be a birational morphism between projective primitive symplectic varieties. Suppose that $X$ has terminal singularities. Then the Galois group $G$ of the finite Galois cover $p:\Def^\lt(X)\to \Def(\widebar X,\pi)$ is isomorphic to
$\Pi_{B\in \mathcal B}W_B$.
\end{theorem}
\begin{proof}
Recall from Section \ref{section folded dynkin} that locally analytically around each point in $\widebar U:=\widebar X \setminus \Sigma_0$ the variety $\widebar X$ is either smooth or has transversal ADE-singularities as in \eqref{eq kaledin germs}. By terminality of singularities, $\pi$ is a resolution over points in $\widebar U$. But an ADE-singularity has a unique crepant resolution and the ADE type is moreover constant at the general point  of a component $B\in \sB$ so that near such a point we can identify $\pi$ with a map
\begin{equation}\label{eq transversal minimal resolution}
\pi_B: S_B\times  \Delta^{2n-2}  \to \widebar S_B \times  \Delta^{2n-2}
\end{equation}
induced by the minimal resolution $S_B\to \widebar S_B$ of an ADE surface singularity $(S_B,0)$. Here, $\Delta^{2n-2}\subset \C^{2n-2}$ is a small polydisk. Let us write $T_B:= S_B\times \Delta^{2n-2}$ and $\widebar T_B:=\widebar S_B \times \Delta^{2n-2}$ and choose such identifications. In particular, we have chosen inclusions $\widebar T_B \subset \widebar U$ and $T_B\subset U:=\pi^{-1}(\widebar U)$ for each $B$.

Each deformation of $\widebar X$ determines upon choosing a lift of the symplectic form $\sigma$ a Poisson deformation of $(\widebar X,\sigma)$ which in turn determines a Poisson deformation of $\widebar T_B$ for each $B$. We refer to \cite{Nam10} and references therein for the notion of a Poisson deformation; we will however formulate as much as possible in the language of ordinary deformations. As explained in \cite[p.~741]{Nam10}, a Poisson deformation  of $\widebar T_B$ is uniquely determined by the underlying flat deformation. Finally, Poisson deformations of $\widebar T_B$ are in bijective correspondence with ordinary flat deformations of $\widebar S_B$. Put together, we obtain a morphism $\Def^\lt(\widebar X) \to \Def(\widebar S_B)$ for each $B$ and similarly for $X$. In fact, we obtain the following commutative diagram 
\begin{equation}\label{eq:Nami24}
\xymatrix{
\Def^{\lt}(X) \ar[r]\ar[d] & \Pi_{B\in \mathcal B}\ \Def (S_B) \ar[d]\\
\Def(\widebar X,\pi) \ar[r] & \Pi_{B\in \mathcal B}\ \Def (\widebar S_B).\\
}
\end{equation}
From diagram (\ref{eq:Nami24}) we deduce the existence of a natural map
\begin{equation}\label{eq:iota}
  \iota  :\Def^{\lt}(X) \to \Def(\widebar X,\pi) \times_{\prod_{B\in \mathcal B}\ \Def (\widebar S_B)}\ \prod_{B\in \mathcal B}\ \Def(S_B),
\end{equation}
which we check to be an isomorphism as in \cite{Nam10}. This is where the projectivity assumption is needed. It allows to deduce that $\Def(\widebar X, \pi)$ is smooth by embedding it in the smooth ambient space $\Def(\widebar X)$,  see \cite[Theorem~1]{Nam06}.

While $\Def(S_B)\to \Def(\widebar S_B)$ is a Galois cover with Galois group $W(\Phi_B)$ (in the notation of Section \ref{section folded dynkin}), see~\cite{Slo80}, over the image of $\Def(\widebar X,\pi)$ it restricts to a cover with Galois group isomorphic to $W_B$ by \cite[Lemma~1.2]{Nam10}. Thus, we conclude that $G=\Pi_{B\in \mathcal B} W_B$ as claimed.
\end{proof}
\begin{remark}\label{rmk:implies}
Notice that Theorem \ref{thm:refl-are-mon}  in particular proves Theorem~\ref{theorem galois cover without group} in the terminal case, by considering  the isomorphism in (\ref{eq:iota}). 
\end{remark}
\begin{remark}\label{rmk:crucial}
What is crucial in the proof of Theorem~\ref{thm:refl-are-mon} is the smoothness of the Kuranishi space $\Def(\widebar X)$. This has been shown by Namikawa for projective symplectic varieties in \cite[Theorem~1]{Nam06} and by Bakker--Lehn  (reducing to Namikawa) in \cite[Proposition~5.24]{BL21} for symplectic varieties admitting an irreducible symplectic resolution. 
\end{remark}


\begin{theorem}\label{theorem reflection}
Let $X$ be a projective primitive symplectic variety with terminal singularities and let $E$ be a prime exceptional divisor on $X$. Then the reflection
\[
R_E: H^2(X,\Q)\to H^2(X,\Q), \quad \alpha \mapsto \alpha -2 \frac{q(E,\alpha)}{q(E,E)} E
\]
is integral and a monodromy Hodge isometry.
\end{theorem}
\begin{proof}
Let $\pi\,:\,Y\to X$ be a $\Q$-factorialization of $X$. Notice that this map is small, hence $\pi^*E$ is still prime exceptional on $Y$. Up to replacing $Y$ with a different birational model, by \cite[Theorem 1.2]{LMP21} we can suppose that $\pi^*E$ can be contracted through a map $c\,:Y\to \widebar{Y}$. With the notation of the previous theorem, $B$ is irreducible and $W_B$ contains only a single non trivial element of order $2$, which by Lemma~\ref{lemma monodromy fixes cohomology downstairs} fixes the hyperplane $H^2(\widebar Y,\Z)$. As by Theorem~\ref{theorem galois cover without group} the representation of $W_B$ on $H^2(Y,\Z)$ is faithful, the nontrivial element is precisely the reflection in $\pi^*E$. In particular, the latter is integral. 

As $\pi^*E$ lies inside the isometric image of $H^2(X,\Z)$ in $H^2(Y,\Z)$ and the BBF-form is compatible with pullback along birational maps, the reflection $R_{\pi^*E}$ restricts to $R_E$ on $H^2(X,\Z)_\tf$. In particular, it leaves $H^2(\widebar X,\Z)$ invariant and is integral on the cohomology of $X$. By Proposition~\ref{prop defo}, we can identify $\Def^\lt(X)$ with a subspace of $\Def^\lt(Y)$ that intersects the branch locus of $p:\Def^\lt(Y)\to \Def(\widebar Y, c)$ transversally. Note that the branch locus of $p$ is just the Hodge locus of $\pi^*E$. Indeed, by Theorem~\ref{thm:refl-are-mon} the map $p$ is a non-trivial branched cover and one is essentially left to show that the contraction exists exactly along the Hodge locus, which again follows from \cite[Corollary~2.29]{BGL20}. Hence, the restriction of the double cover $p$ to $\Def^\lt(X)$ is still a nontrivial double cover so that restriction gives an isomorphism of Galois groups. This implies that $R_E$ is also a monodromy operator. The fact that it is a Hodge isometry follows from Theorem \ref{theorem galois cover without group}.
\end{proof}

\begin{remark}\label{remark dual curve}
Using adjunction, we can also see that the class of the general curve $R$ in the ruling of $E$ interpreted as an element in $H_2(X,\Q)^\vee$ is equal to $-2 \frac{q(E,\bullet)}{q(E,E)}$. This uses the terminality of $X$ (which implies that the  $R$ does not meet $X^\sing$) and the fact that the normal bundle of the curve inside of $E$ is trivial.
\end{remark}


One can ask whether the hypothesis on $X$ having terminal singularities in Theorem~\ref{thm:refl-are-mon} can be dropped. 
The following simple example shows that even the reflections generated by the exceptional divisors are not in general in the Galois group, i.e. there is no analog of Theorem~\ref{theorem reflection} without the terminality hypothesis.

\begin{example}\label{example:counterex}
Let $\wt S$ be a K3 surface containing an $A_2$ string of $(-2)$-curves $C_1, C_2$ and let $\wt S \to S \to \widebar S$ be the successive contractions of $C_1$ in the first and the image of $C_2$ in the second step. Let us consider the induced maps between the full Kuranishi spaces 

\[
\xymatrix{
\Def(\wt S) \ar[r]_{p_1} \ar@/^6mm/[rr]^p & \Def(S) \ar[r]_{p_2} & \Def(\widebar S).
}
\]
Then $\Gal(p)=S_3$ and $\Gal(p_1)=S_2$, the Weyl groups of the $A_2$ and $A_1$ root systems. The map $p_2$ consequently has degree $3$ but we will see that it becomes an isomorphism upon restriction to $\Def^\lt(S)$. 
Let us describe the action of the Weyl group $S_3$ in the period domain. There, the action is generated by $r_1, r_2$, the reflections in $C_1,C_2$. We may identify $\Def^\lt(S)$ with $C_1^\perp$, the Hodge locus of $C_1$. This is pointwise fixed by $r_1$ but it is not fixed by $r_2$ because $C_1$ and $C_2$ are not perpendicular. In fact, the $S_3$-orbit of $C_1^\perp$ also contains $C_2^\perp$ and $(C_1+C_2)^\perp$. We see that under the quotient map $p$, the space $\Def^\lt(S)$ is mapped isomorphically onto its image. Hence, $\Def^\lt(S)\to \Def(\widebar S,\pi)$ is an isomorphism where $\pi:S\to \widebar S$ is the contraction of $C_2$. 

Let $D_2$ denote the image of $C_2$ in $S$. The fact that $r_2$ does not preserve $C_1^\perp$ is equivalent to the reflection $\widebar r$ in $D_2$ on the second cohomology of $S$ not being integral. Indeed, $D_2$ is 2-factorial and $2D_2$ is primitive with square $(-6)$ and divisibility $1$. For an explicit example, take e.g. a smooth quartic surface $\wt S$ containing two incident lines $C_1,C_2$ and let $H$ be the hyperplane class. Then, with the notation above, $r_{D_2}(H)= H+\frac{D_2}{3}$. 
In particular, the analog of Theorem~\ref{theorem reflection} cannot hold for $S$. 
\end{example}




\section{Monodromy and cones}\label{section wall divisors}

\subsection{Monodromy operators induced by birational maps}\label{section bir monodromy}
We will show in this paragraph that birational maps between $\Q$-factorial terminal projective primitive symplectic varieties induce monodromy operators. Note that for those varieties, a birational map induces a morphism between the second cohomology groups (by $\Q$-factoriality)  which is an isomorphism (by terminality, cf. \cite[Corollary 3.54]{KM98}).

\begin{theorem}\label{theorem birational monodromy}
Let $\phi:X\ratl X'$ be a birational map between  projective primitive symplectic varieties with $\Q$-factorial terminal singularities. Then $\phi_* \in \Mth(X)$.
\end{theorem}
\begin{proof}
By \cite[Theorem~6.16]{BL18}, there are one-parameter deformations $\scrX$, $\scrX'$ of $X$, $X'$ over a disk $\Delta$ such that $\scrX$ and $\scrX'$ are birational over $\Delta$ such that the birational map is an isomorphism over the pointed disk $\Delta^\times$. In particular, $X$ and $X'$ are inseparable in the moduli space. It follows from the proof that $\scrX$, $\scrX'$ are projective over $\Delta$ and the birational map between them induces the map $\phi$ on the central fiber. Let us consider the graph $\Gamma\subset \scrX\times_\Delta\scrX'$ of the birational map. Its central fiber contains the graph of $\phi$ but possibly also other things, i.e.
\[
\Gamma_0=\Gamma_\phi + \sum_{i} Y_i
\]
where $Y_i \subset X\times X'$ are cycles of dimension $\dim X$. By construction, the correspondence $\Gamma_0$ acts on $H^2(X,\Q) \to H^2(X',\Q)$ by a monodromy operator. By \cite[Corollary~5.2 item (ii)]{Huy99}, the $Y_i$'s map to subvarieties of positive codimension in $X$ and $X'$ and only if they map to a divisor, they contribute to the action of the correspondence $[\Gamma_0]_*$ on $\Pic_\Q(X)$. 
The argument in \cite[Proof of Theorem 2.5]{Huy03} applies {\it verbatim} here and shows precisely that the $Y_i$'s map to codimension $\geq 2$. The conclusion follows.
\end{proof}

With this result, the following definition makes sense.

\begin{definition}
Let $X$ be a projective primitive symplectic variety. We denote by $\Mtb(X)$ the subgroup of $\Mth(X)$ consisting of all monodromy operators induced by birational maps.
\end{definition}

\subsection{Cones in the real Picard group}\label{section cones}

Both the birational and the ordinary cone conjecture are statements about cones in the real Picard group which we now define. 

\begin{definition}
Let $X$ be a projective primitive symplectic variety and let $\Pic(X)_\R:=\Pic(X)\tensor \R$ be the real Picard group. Inside $\Pic(X)_\R$ we consider the following cones:
\begin{enumerate}
    \item The {\it ample cone} $\Amp(X)$ of $X$, i.e. the cone generated by all ample (integral) Cartier divisors on $X$. 
    \item The {\it nef cone} $\Nef(X)$ of $X$, i.e. the closure of the ample cone. Note that this is usually strictly bigger than the cone generated by all nef (integral) divisors).
    \item The {\it movable cone} $\Mov(X)$ of $X$, i.e. the cone generated by all effective (integral) divisors without fixed part. This cone is neither closed nor open in general. We denote by $\Movbar(X)$ its closure and by $\Movint(X)$ its interior. 
    \item The {\it positive cone} $\Pos(X)$ of $X$ which is the connected component containing $\Amp(X)$ of the cone of positive vectors of $\Pic(X)_\R$ with respect to the BBF form.
    \item If $X$ has $\Q$-factorial terminal singularities, we define the {\it birational ample cone} $\BAmp(X)$ of $X$ to be the the union of all $f^*\Amp(X')$ where $f:X\ratl X'$ is a birational map to another primitive symplectic variety $X'$ with $\Q$-factorial terminal singularities. Similarly, one defines $\BNef(X)$.
    \item The {\it effective cone} $\Eff(X)$ of $X$ which is the cone spanned by all effective divisors and the {\it big cone} $\Big(X)$ which is the cone spanned by all big divisors.
\end{enumerate}
\end{definition}

We  have the following inclusions:
\begin{equation}\label{eq cones inclusions}
    \Amp(X) \subset  \BAmp(X) \subset\Movint(X) \subset \Pos(X).
\end{equation}
The lefthand side inclusion is clear. 
The middle inclusion follows from the fact that birational maps between terminal $K$-trivial varieties are isomorphisms in codimension one (cf. \cite[Corollary 3.54]{KM98}). 
The righthand side inclusion is given in particular by the following lemma, which we record here also for later use. 

\begin{lemma}\label{lem:mov-pos}
Let $X$ be a projective primitive symplectic variety.
If $D$ is a divisor in the interior of the movable cone of $X$, 
then  $q_X(D, D')> 0$ for all $D'$ prime divisors on $X$.
\end{lemma}
\begin{proof}
To start with recall that the Boucksom-Zariski decomposition holds for any effective $\Q$-divisor $D$ on primitive symplectic varieties, by \cite[Theorem 1.1]{KMPP19}. Moreover if $D=P+N$ is such a decomposition, 
for all integers $k>0$ such that $kP$ and $kD$ are integral,
we have an isomorphism 
$$
 H^0(X,\mathcal O_X(kP))\cong H^0(X,\mathcal O_X(kD)),
$$
which implies that 
\begin{equation}\label{eq:star}
    N \text{ coincides with the divisorial base locus of } D.
\end{equation}

Therefore if $D$ is a movable line bundle on $X$, 
then $D$ is $q_X$-nef (i.e. $q_X(D, D')\geq 0$ for all $D'$ prime divisors on $X$).
Now suppose that $q_X(D,D')=0$ for some prime divisor $D'$. Since $D$ lies in the interior of the movable cone, for $0<\varepsilon \ll 1$ and $A$ an ample class, then $D-\varepsilon A\in \Mov(X)^\circ$ which yields the following contradiction  
$$
q_X(D-\varepsilon A, D')= -\varepsilon q_X( A, D')<0.
$$
\end{proof}
The following lemma adds another inclusion to  chain (\ref{eq cones inclusions}).

\begin{lemma}\label{lemma pos in big}
Let $X$ be a projective primitive symplectic variety. Then $\Pos(X) \subset \Big(X)$.
\end{lemma}
\begin{proof}
Since $\Pos(X)$ is open and convex, it is enough to show that rational, even integral classes in $\Pos(X)$ are big. Let $\alpha$ be such a class and let $L_\alpha$ be the corresponding line bundle. Take a general one parameter deformation in the Hodge locus of $L_\alpha$. Then on a nearby fiber $X_t$ with Picard rank one, the corresponding line bundle $L_{t}$ is ample by the projectivity criterion \cite[Theorem~6.9]{BL18}. In particular, the dimension of $H^0(X_t,L_t^{\tensor m})$ grows as a polynomial of degree $\dim X$ in $m$. By semicontinuity, sections of powers of $L_\alpha^{\tensor m}$ grow at least as much and thus $\alpha \in \Big(X)$.
\end{proof}

Furthermore, let us observe that on any projective variety $X$ the big cone is contained in the effective cone 
\begin{equation}\label{eq big in eff}
    \Big(X) \subset \Eff(X).
\end{equation}

\subsection{A Hodge-theoretic global Torelli}\label{ss:torelli}

In the proof of the decomposition of the monodromy group $\Mth(X)$ as a semidirect product of $\Mtb(X)$ with the subgroup of reflections in prime exceptional divisors we will need the following Hodge-theoretic version of the global Torelli theorem for projective primitive symplectic varieties with $\Q$-factorial terminal singularities. For the analogous statement in the smooth case, see \cite[Theorem 1.3, item (2)]{Mar11}. 

\begin{theorem}\label{thm:hodge-torelli}
Let $X$ and $X'$ be projective primitive symplectic varieties with $\Q$-factorial terminal singularities such that they are equivalent under locally trivial deformations.
Let 
$$
p:H^2(X,\Z) \to H^2(X',\Z)
$$  
be a locally trivial parallel transport operator which is a Hodge isometry.  Assume $b_2(X) \geq 5$. 
Then there is a birational map $f:X\ratl X'$ with $f_*=p$ if and only if  the restriction  of $p$ to $\Pic(X)$ satisfies
\begin{equation}\label{eq:equiv}
  p(\BAmp(X))\cap \BAmp(X') \not= \emptyset.
\end{equation}
\end{theorem}
\begin{proof}
The ``only if'' direction follows from Theorem~\ref{theorem birational monodromy}. 
Let $\varphi$ be a marking of $X'$. By hypothesis, $(X',\varphi)$ and $(X,  \varphi\circ p)$ lie in the same connected component of the marked moduli space and have the same period. Hence, by \cite[Theorem~1.1]{BL18} (which uses the hypothesis $b_2(X) \geq 5$), the varieties $X$ and $X'$ are birational so that by \cite[Theorem~6.16]{BL18}, there are one parameter deformations $\scrX$ and $\scrX'$ of $X$ and $X'$ over a disk $\Delta$ which are isomorphic over the pointed disk and such that the graph $\Gamma_t \subset \scrX_t\times \scrX'_t$ of this isomorphism for $t\neq 0$ degenerates to a cycle
\[
\Gamma_0 = Z + \sum_i Y_i \subset X \times X'
\]
where the correspondence $[\Gamma_0]_*:H^2(X,\Q) \to H^2(X',\Q)$ coincides with $p$ and $Z$ is the graph of a birational map $f:X \ratl X'$. 
Again, as in the proof of Theorem \ref{theorem birational monodromy}, we invoke the argument spelled out in
 \cite[Proof of Theorem 2.5]{Huy03} to deduce that the $Y_i$'s map to codimension $\geq 2$. Hence $[\Gamma_0]_*=p=f_*$ and the conclusion follows.
\end{proof}
Notice that if $f:X\dashrightarrow X'$ is the birational map such that $f_*=p$ given by Theorem~\ref{thm:hodge-torelli} and $p(\Amp(X))\cap \Amp(X')\neq \emptyset$, then $f$ must be an isomorphism.


\section{The semidirect product decomposition}\label{section semidirect}

Let $X$ be a projective primitive symplectic variety with $\Q$-factorial, terminal singularities. The main goal of this section is to show that the group $\Mth(X)$ of Hodge monodromy operators is equal to the semidirect product of the normal subgroup generated by reflections in prime exceptional divisors (cf. Theorem \ref{theorem reflection}) and of the subgroup $\Mtb(X)$ of {monodromy
operators induced by birational maps from $X$ to itself} (cf. Theorem \ref{theorem birational monodromy}). 

From such result, as it is done in the smooth case in \cite[Section~6]{Mar11}, we will deduce the birational cone conjecture in Section \ref{sec:bircone}. We follow Markman's approach \cite[Sections~5 and 6]{Mar11} and using our results from previous sections we arrive at the proof of the cone conjecture. 

\begin{definition}\label{def:se}
A line bundle $L\in \Pic(X)$ is called {\it stably exceptional} if there exists a closed analytic subset $Z\subset \Def(X,L)$ of positive codimension such that for all
$t\in \Def(X,L)\setminus Z$ the line bundle $L_t$ equals $\mathcal O_{X_t}(a E_t)$ where  
$E_t$ is a prime exceptional divisor on $X_t$ and $a\in \mathbb Z_{>0}$ the Cartier index  of $E_t$.
\end{definition}

Notice that on special points (i.e. along $Z$) a stably exceptional line bundle $L$ may have $h^0 >1$, see Example~\ref{example stably exceptional degenerates}. Moreover, even when $h^0 =1$ the divisor may be reducible or non-reduced, i.e. not {\it prime} exceptional, and for this reason we prefer not to use the terminology {\it stably prime exceptional} introduced in \cite[Definition 6.4]{Mar11} in the smooth case. It is important to notice that the class of a stably exceptional line bundle is in the effective cone by the semicontinuity of $h^0$.

\subsection{Moduli spaces of primitive symplectic varieties}\label{section moduli}

 Given a lattice $\Lambda$ with quadratic form $q$, a $\Lambda$-marking of $X$ is an isomorphism $\phi : (H^2(X, \Z)_\tf , q_X ) \cong (\Lambda, q)$.
 
 By \cite[Theorem 4.11]{BL18} the space $\Def^\lt(X)$ is smooth, and by patching together these smooth local charts of locally trivial Kuranishi families via the local Torelli theorem proved in \cite[Corollary 5.8]{BL18} we can define (exactly as in the smooth case) the analytic coarse moduli space of $\Lambda$-marked primitive symplectic varieties. We will denote it by $\mathfrak M^{\rm lt}_\Lambda$.

\begin{proposition}\label{prop:prime-is-se}
Let $X$ be a projective primitive symplectic variety with terminal and $\Q$-factorial singularities and $E\subset X$ a prime exceptional divisor. Then
\begin{enumerate}
\item[(1)] The line bundle $\mathcal O_X(E)$ is stably exceptional. 
\item[(2)] Let $X'$ be a primitive symplectic variety, locally trivial deformation equivalent to $X$ and $g:H^2(X,\mathbb Z)\to H^2(X',\mathbb Z)$ a locally trivial parallel trasport operator which is a Hodge isometry. Set $\alpha:=g([E])\in H^{1,1}(X',\mathbb Z)$. Then either $\alpha$ or $-\alpha$ is the class of a stably exceptional line bundle on $X'$. 
\end{enumerate}
\end{proposition}
\begin{proof}
By \cite[Theorem 1.2]{LMP21}, the divisor is uniruled, and if $R$ is a general curve in the ruling of $E$ we have that 
$R^\vee \in \Q_{>0} [E]$ (in particular, the two Hodge loci coincide $\Hdg_{[R]}=\Hdg_{[E]}$), the curve $R$ deforms along its Hodge locus and for all $t\in \Hdg_{[R]}$ the variety $X_t$ contains a uniruled divisor $E_t$ ruled by deformations $R_t$ of $R$. By \cite[Theorem 1.2, item (2.b)]{LMP21}, the divisor $E_t$ is prime on an open set of the Hodge locus, so that $\mathcal{O}_X(E)$ is stably prime exceptional and item (1) holds.

(2) Let now $\varphi:H^2(X,\mathbb Z)\to \Lambda$ be a marking and $e:=\varphi([E])$. 
Let $\mathfrak M\subset \mathfrak M^{\rm lt}_\Lambda$ be a connected component of the  locally trivial $\Lambda$-marked moduli space (defined in Section~\ref{section moduli}) that contains 
$(X,\varphi)$ and let $\mathfrak p$ be the period map. 
Consider the marking $\varphi'$ on $X'$ given by 
$$
 \varphi':= \varphi \circ g^{-1}.
$$ 
Notice that by construction $(\varphi')^{-1}(e)= \alpha=g([E])$. We deduce that
$$
 (X',\varphi')\in (\mathfrak p_{|\mathfrak M})^{-1}(e^\perp)= \Hdg_{[E]}. 
$$
For general $t\in \Hdg_{[R]}$ we have $\rho(X_t)=1$ and again the deformation $R_t$ of $R$ verifies $R_t^\vee = \Q_{>0} [E_t]$.
Therefore, on such $X_t$ any parallel transport of $\alpha$ has a non-zero (possibly negative) multiple which is exceptional. 
\end{proof} 
By Theorem~\ref{theorem reflection}, reflections $$R_\ell : \alpha\mapsto \alpha - 2\frac{q_X(\alpha,\ell)}{q_X(\ell)}\ell
$$ 
by classes $\ell$ of prime exceptional line bundles are monodromy Hodge isometries. As stably exceptional line bundles are exceptional on a small deformation insider their Hodge locus, parallel transport gives that $R_\ell$ is already a monodromy Hodge isometry if $\ell$ is only \emph{stably exceptional}. Here we verify the following.

\begin{lemma}\label{lem:normality}
The subgroup $W_X \subset \Mth(X)$ generated by  reflections 
by classes $\ell$ of stably exceptional line bundles is normal.
\end{lemma}
\begin{proof}
Let $R_\ell $ be a reflection by a stably exceptional class $\ell$. 
We take $g\in \Mth(X)$ and consider the image of a cohomology class $\alpha$ under the monodromy operator $gR_\ell g^{-1}$. Set $\beta:=g^{-1}(\alpha)$.
Notice that as $g$ is an isometry we have
$$
q_X(g^{-1}(\alpha), \ell)= q_X(\alpha, g(\ell)). 
$$
Therefore we deduce the following:
$$
 \alpha \mapsto g(R_\ell(g^{-1}(\alpha)))=g(\beta -  2\frac{q_X(\beta,\ell)}{q_X(\ell)}\ell)=\alpha - 2\frac{q_X(\alpha,g(\ell))}{q_X(g(\ell))}g(\ell)= R_{g(\ell)}(\alpha).
$$
Since by Proposition \ref{prop:prime-is-se}, item (2), $g(\ell)$ or its opposite is a stably exceptional class, the conclusion now follows. 
\end{proof}

\begin{definition}\label{def:ex-chamber}
Let $X$ be a primitive symplectic variety and $q_X$ its BBF quadratic form. 
\begin{enumerate}
\item The normal subgroup $W_X\subset \Mth(X)$ from Lemma~\ref{lem:normality} will be referred to as the \emph{Weyl group} of $X$.
\item The {\it fundamental exceptional chamber} $\FE X$ is the subcone of classes $\alpha \in \Pos(X)$ such that $q_X(\alpha, [E])>0$ for every prime exceptional divisor $E\subset X$. 
\item The {\it fundamental uniruled chamber} $\FU X$ is the subcone of classes $\alpha \in \Pos(X)$ such that $q_X(\alpha, [D])>0$ for every uniruled  divisor $D\subset X$. 
\item The {\it stably exceptional chamber} $\SE X$ is the subcone of classes $\alpha \in \Pos(X)$ such that $q_X(\alpha, \ell)>0$ for every class $\ell\in \Pic(X)$ of a stably exceptional line bundle on $X$. 
\item An {\it  exceptional chamber} of the positive cone $\Pos(X)$ is a subset of the form $g \cdot \FE X$, for some $g\in \Mth(X)$.
\end{enumerate}
\end{definition}

Recall from \cite[Theorem~1.1]{KMPP19} that (in particular) every effective $\Q$-Weil divisor $D$ on a primitive symplectic variety has a Zariski decomposition $D=P(D)+N(D)$ where $P(D)$ is $q_X$-nef, i.e. it pairs nonnegatively with all effective Cartier divisors. 
We need the following property of the positive part.

\begin{proposition}\label{proposition kmpp}
Let $D$ be an effective $\Q$-Cartier divisor on a projective primitive symplectic variety $X$ and let $D=P(D)+N(D)$ be its Zariski decomposition in the sense of \cite{KMPP19}. If $D\in \Pos(X)$, then $P(D)$ is movable.
\end{proposition}
\begin{proof}
As by \cite[Theorem~1.1]{KMPP19} the divisor $P(D)$ is an effective $\Q$-Cartier with $q_X(P(D)) \geq q_X(D)$, we may assume $D=P(D)$. Notice that  only flips appear in any $K_X+D=D$ log-MMP on $X$. Indeed suppose by contradiction that  a $D$-negative curve $R$ covers a divisor $E$. By the cone theorem $R$ can be contracted. By Corollary \ref{corollary curve deforms} the curve $R$ deforms along its Hodge locus, at the general point $t$ of which the Picard number $\rho(X_t)$ is one. This means that the deformation $R_t$ of $R$ has class dual to a (positive) multiple of $E_t$, the deformation of $E$. In conclusion we have that $R=\lambda E^\vee, \lambda >0$.
Then $0>R\cdot D=\lambda q_X(E,D)\geq 0$, contradiction. 
By \cite[Corollary~1.4.2]{BCHM10}, any MMP with scaling for an appropriate ample divisor terminates and there is a birational model $\phi:X\ratl X'$ of $X$ such that $\phi$ is an isomorphism in codimension one and $D':=\phi_*D$ is nef on $X'$. Moreover, $D'$ is big by Lemma~\ref{lemma pos in big} so by Kawamata's base-point-free theorem \cite[Theorem~6.1]{Kaw85} we obtain that $D'$ is semi-ample. In particular, $D'$ is movable and, since $\phi$ is an isomorphism in codimension one, so is $D$.
\end{proof}

By passing to a $\Q$-factorialization \cite[Corollary~1.4.3]{BCHM10}, we could also make sense of the previous statement for $\Q$-Weil divisors. As we do not need this case in what follows, we have decided to restrict our attention to $\Q$-Cartier divisors. Note that if the hyperkähler SYZ-conjecture holds for primitive symplectic varieties, the positivity assumption in Proposition~\ref{proposition kmpp} is superfluous. See \cite[Section~21.4]{GHJ03} or \cite[Conjecture 4.1]{Saw03} for the statement and a discussion of the conjecture in the smooth case. With the same notation as above, we have:

\begin{corollary}\label{corollary kmpp}
Let $D$ be an effective $\Q$-Cartier divisor on $X$. Then $P(D)\in\Movbar(X)$.
\end{corollary}
\begin{proof}
We have $P(D)\in \widebar{\Pos(X)}$, so applying Proposition~\ref{proposition kmpp} to $D+\frac{1}{n}A$ for some ample divisor $A$ we obtain the result as $n\to \infty$.
\end{proof}

\begin{proposition}\label{prop:fe-in-bdp}
Let $X$ be a projective primitive symplectic variety with $\Q$-factorial terminal singularities. 
 Then 
$$
 {\BAmp(X)}\subset \FE X \subset \FEbar X \subset \widebar{\BAmp(X)} = \Movbar(X).
$$
\end{proposition}
\begin{proof}
For the first inclusion consider $\alpha\in {\BAmp(X)}$ and a prime exceptional divisor $E$ on $X$. Then there is a birational $f:X \ratl X'$ to a projective primitive symplectic variety $X'$ with $\Q$-factorial terminal singularities such that $f_*\alpha$ is in the ample cone of $X'$. Then $f_* E$ is still effective on $X'$ and we are reduced to showing that an effective divisor pairs positively with an ample divisor. This immediately follows from the usual formulas for the BBF form, see e.g. \cite[Exercise~23.2]{GHJ03} which is derived the same way as in the smooth case. Another way of proving this is to use the Boucksom-Zariski decomposition in the singular case \cite[Theorem 1.1]{KMPP19}, from which one deduces that the intersection cannot be $<0$. But it cannot be zero either by the Hodge index theorem and the conclusion follows. 
Let us prove the equality in the above chain. Clearly, a divisor which is ample on a birational model is movable. It suffices to show that an element in the interior of $\Mov(X)$ is nef on a birational model of $X$. Note that such an element is $q_X$-positive by \eqref{eq cones inclusions}. The nefness on a birational model is seen as in the proof of Proposition~\ref{proposition kmpp}.

An element of $\FEbar X$ is a limit of rational vectors in $\FE X$. Indeed, an element of $\FE X$ can be made rational by adding arbitrarily small classes in $\Amp(X)$ which by the first inclusion and convexity does not take us out of $\FE X$. It thus suffices to show that each rational vector in $\FE X$ already lies in $\Movbar(X)$. Let $D$ be such a class. By definition, $\FE X \subset \Pos(X)$, hence $D$ is effective by Lemma~\ref{lemma pos in big} and \eqref{eq big in eff}. We may therefore apply \cite[Theorem~1.1]{KMPP19} to write $D=P(D)+N(D)$ where $P(D)$ is $q_X$-nef, $N(D)$ is $q_X$-exceptional, and the decomposition is $q_X$-orthogonal. In particular, $N(D)$ is a sum of prime exceptional divisors, so from the assumption $D\in \FE X$ we deduce
\[
0 \leq q_X(D,N(D))= q_X(N(D)).
\]
But $q_X(N(D)) \leq 0$ with equality if and only if $N(D)=0$, cf. \cite[Theorem~1.1]{KMPP19}. Hence $D=P(D)$ and it is movable by Proposition~\ref{proposition kmpp}. The desired inclusion follows.
\end{proof}


\begin{proposition}\label{prop:exc=sexc}
Let $X$ be a projective primitive symplectic variety with terminal and $\Q$-factorial singularities. Then $\FE X =\SE X$.
\end{proposition}
\begin{proof}
 First of all by Proposition \ref{prop:prime-is-se}, item (1), we have 
$$
 \SE X\subset \FE X. 
$$
Conversely, let $\alpha\in \FE X$. For any stably exceptional class $\ell$, we consider its Boucksom-Zariski decomposition (whose existence in the singular case is ensured by \cite[Theorem 1.1]{KMPP19}):
$$
\ell = P(\ell) + N(\ell). 
$$
Notice that the negative part $N(\ell)$ cannot be zero since  $q_X(\ell)<0$ while $q_X(P(\ell))\geq 0$. Moreover, $N(\ell)$ is the sum of prime exceptional divisors so that $q_X(\alpha,N(\ell)) >0$ by the hypothesis that $\alpha\in \FE X$. We also have $q_X(\alpha,P(\ell))\geq 0$ as both classes belong to the closure of $\Pos(X)$. Putting both inequalities together, we conclude $q_X(\alpha,\ell)\geq q_X(\alpha,N(\ell)) >0$ and hence $\alpha\in\SE X$.
\end{proof}

\begin{example}\label{example stably exceptional degenerates}
Note that a stably exceptional class does not always specialize to a sum of prime exceptional divisors (if this were the case, the equality $\FE E=\SE X$ would be true for trivial reasons). As a counterexample we may take an elliptic K3 surface with a section and an $I_2$-fiber (i.e. two smooth rational curves meeting in two distinct points). Let $s$ be the class of the section, $f$ the class of the general fiber, and $e$ the class of the rational curve in the $I_2$-fiber that meets the section. We have
\[
e^2=s^2=-2, \quad f^2=0=e.f, \quad f.s=1=e.s.
\]
We deduce that $\alpha:=s+2e+2f$ is a primitive class of square $-2$. Generically on its Hodge locus, it is effective and irreducible. Thus, $\alpha$ is stably exceptional. If   $E_t$ are the $(-2)$-curves on general nearby fibers in the Hodge locus of $\alpha$, then $E_t$ specializes to an element in the linear system corresponding to $\alpha$. If $L_\alpha$ denotes the corresponding line bundle, one verifies that $H^0(X,L_\alpha)=H^0(X,L_{2f})$, so every specialization of $E_t$ will have a summand with square zero. By moving to a nearby point in the common Hodge locus of $s+2f$ and $e$, one obtains an example where the specialization always contains a square positive summand. In both cases, we see that a stably exceptional bundle may have a higher-dimensional space of sections at special points of its Hodge locus.
\end{example}

Let us put
\begin{equation}
\mathcal U:= \Pos(X) \setminus \bigcup \{ \ell^\perp : \ell\ {\rm stably\ exceptional\ class}\}.
\end{equation}
Consider now a real vector space endowed with an inner product as follows 
\[
V:=\left(\R^{r+1}, (\cdot,\cdot)\right),  \quad \textrm{ where } \quad  (x,x):=x_0^2-\sum_{i>0}x_i^2
\]
and 
\[
\mathbb H^n:=\{x\in V  \mid  x^2=1,\ x_0>0\}.
\]
The key to study the monodromy group are the results in hyperbolic geometry due to Vinberg and Shvartsman \cite{vinberg} on 
discrete groups of motions (i.e. subgroups of $\O^+(V)$ with finite stabilizers and discrete orbits) of $\mathbb H^n$ generated by reflections. They prove the existence of a wall and chamber decompositions of $\mathbb H^n$ associated to reflections such that each chamber is a generalized convex polyhedron\footnote{Here, a \emph{generalized} convex polyhedron is the intersection of a countable number of half spaces.} which is a fundamental domain for a hyperbolic reflection group. These results are naturally applied to the case $V:=H^{1,1}(X,\Z)\otimes \R$ once we have the relevant results Theorem \ref{thm:hodge-torelli}, Lemma~\ref{lem:normality}, Propositions~\ref{prop:fe-in-bdp} and~\ref{prop:exc=sexc}. 
We refer the reader to \cite[Section 6.3]{Mar11} for further details.  
\begin{theorem}\label{thm:poly}
The fundamental exceptional chamber $\FE X$ is equal to the connected component of $\mathcal U$ containing the ample cone. 
In particular, $\FE X$ is the interior of a generalized convex polyhedron. 
\end{theorem}
\begin{proof}
Thanks to results due to Vinberg and Shvartsman, see \cite[Theorem 6.15]{Mar11}, the (open) set $\mathcal U$ coincides with the union of the interiors of the fundamental chambers of the action of the Weyl group $W_X$ on $\Pos(X)$. 
The fundamental exceptional chamber $\FE X$ is by definition a convex cone and by Proposition \ref {prop:exc=sexc} it is contained in $\mathcal U$. Hence 
$\FE X$ must be one of the connected components of $\mathcal U$. Since by definition $\FE X$ contains the ample cone, the conclusion follows. 
\end{proof}
We are now ready to prove the semi-direct product decomposition of $\Mth(X)$ along the lines of \cite[Theorem 6.18]{Mar11}. We need the hypothesis $b_2(X)\geq 5$ to use Theorem \ref{thm:hodge-torelli}, which is fundamental for item (4).

\begin{theorem}\label{thm:semidirect}\
Let  $X$ be a projective primitive symplectic variety with $\Q$-factorial and terminal singularities and $b_2(X)\geq 5$. 
\begin{enumerate}
\item The group $\Mth(X)$ acts transitively of the set of exceptional chambers and the normal subgroup $W_X$ acts simply transitively on this set.
\item The exceptional chambers coincide with the connected components of $\mathcal U$, i.e. each exceptional chamber is the interior of a fundamental domain for the $W_X$-action on $\Pos(X)$.
\item The normal subgroup $W_X$ is generated by reflections $R_e$ with respect to classes $e$ of prime exceptional divisors.  
\item The subgroup of $\Mth(X)$ stabilizing the fundamental exceptional chamber $\FE X$ is equal to $\Mtb$.
\item The group $\Mth(X)$ is equal to the semidirect product of the normal subgroup $W_X$ and of $\Mtb$.
\end{enumerate}
\end{theorem}
\begin{proof}
(1): The transitivity of the $\Mth$-action on the set of the exceptional chambers follows from the very definition of these chambers. The $W_X$-action on the set of the exceptional chambers
is both transitive and free by Vinberg and Shvartsman (see \cite[Theorem 6.15]{Mar11}). 

(2): As observed by Markman in  \cite[Proof of Theorem 6.17]{Mar11}, the (open) set $\mathcal U$ coincides with the union of the interiors of the fundamental chambers of the $W_X$-action on $\Pos(X)$. 
The conclusion follows. 

(3): The fundamental exceptional chamber $\FE X$ is the interior of a chamber of $W_X$ by Theorem \ref{thm:poly}. The walls of the boundary of  $\FE X$ are by definition of the form $[E]^\perp\cap \Pos(X)$, for some prime exceptional divisor $E$ on $X$. Then, the conclusion follows from \cite[Theorem 6.15]{Mar11}, item (2). 


(4): Let $f_*\in \Mtb(X)$.  Then $f_* \FE X\subset \FE X$ since $f_*$ induces a bijection between the sets of exceptional divisors. 

For  the opposite implication we argue as follows. Let $g\in \Mth(X)$ be an operator preserving the fundamental exceptional chamber $\FE X$. Let $\alpha_1\in {\FEbar X}$ a very general class. 
By Proposition \ref{prop:fe-in-bdp}, both $\alpha_1$ and the class $\alpha_2:=g(\alpha_1)$ lie in $\widebar{\BAmp(X)}$. 
Therefore, for $i=1,2$, there exists a bimeromorphic map $f_i:X\dashrightarrow X'_i$ and a class $\beta_i\in \Amp(X'_i)$
such that $g(\alpha_i)=f^*_i(\beta_i)$. Then 
$$
(f_2^{-1})^*\circ g\circ f_1^*: H^2(X'_1,\Z) \to H^2(X'_2,\Z)
$$
is a Hodge isometric parallel transport operator by Theorem~\ref{theorem birational monodromy} sending $\beta_1$ to $\beta_2$. Hence, by Theorem~\ref{thm:hodge-torelli}  there exists an isomorphism 
$h:X'_1\cong X'_2$ such that
$$
h_*= (f_2^{-1})^*\circ g\circ f_1^*.
$$
From this we deduce that 
$$
 g= ((f_2)^{-1}\circ h\circ f_1 )_*\in \Mtb(X).
$$
(5): By items (1) and (4) above $\Mth(X)$ is generated by $W_X$ and $\Mtb(X)$. The normality of $W_X$ has been checked in Lemma \ref{lem:normality}. Finally the intersection 
$\Mtb(X)\cap W_X$ is trivial since by item (1) the action of $W_X$ is free. 
\end{proof}

Using the global Torelli theorem proved in \cite{BL18} as well as the non-separability of bimeromorphic IHS proved therein it is possible to generalize
Theorem \ref{thm:semidirect}, item (5). 

\begin{theorem}\label{thm:semidirect-gen}
Let $X_1$ and $X_2$ be projective primitive symplectic varieties with $\Q$-factorial terminal singularities and assume that $b_2(X_1)\geq 5$.
Let $f:H^2(X_1,\mathbb Z)\to H^2(X_2,\mathbb Z)$ be a Hodge isometry. Then there exists a unique $w\in W_{X_2}$ and a birational map 
$h:X_1\ratl X_2$ such that $f=w\circ h_*$. Moreover the map $h$ is uniquely determined, up to composition with an automorphism of $X_1$ acting trivially on $H^2(X_1,\mathbb Z)$.
\end{theorem}
\begin{proof}
By hypothesis $X_1$ and $X_2$ belong to the same connected component of the moduli space ${\mathfrak M}^{\lt}_\Lambda$ and have the same 
period. Therefore, by \cite[Theorem 1.2]{BL18} there exists a birational map 
$g:X_1\dashrightarrow X_2$. Here we use the hypothesis $b_2(X_1)\geq 5$. The map $g$  induces a parallel trasport operator $g^*: H^2(X_2,\mathbb Z)\to H^2(X_1,\mathbb Z)$ which is a Hodge isometry by \cite[Corollary 6.18]{BL18}. Then $f\circ g^*\in \Mth(X_2)$ and by Theorem~\ref{thm:semidirect}, item (5), there exists a unique $w\in W_{X_2}$ such that $w^{-1}f\circ g^*\in \Mtb(X_2)$. 
Let now $\phi:X_2\dashrightarrow X_2$ the birational map such that $\phi_* =w^{-1}f\circ g^*$. Then setting $h:=\phi\circ g$ we are done. 

Suppose now that $h':X_1\dashrightarrow X_2$ is a birational map and $w'\in W_{X_2}$ such that $f=w'\circ h'_*$.
Then $w^{-1}w'= ((h')^{-1}h )_*\in W_{X_2}\cap \Mtb(X_2)=\{\id\}$, which shows that $w=w'$ and $h'_*=h_*$. Therefore $h^{-1}\circ h'$ is a birational transformation of $X_1$ inducing
the identity on $H^2(X_1,\mathbb Z)$. In particular $h^{-1}\circ h'$ preserves the ample cone, hence it must be biregular. 
\end{proof}

\section{Proof of the birational cone conjecture}\label{sec:bircone}
In this section we prove the $\Movbar^+$-version of Morrison's movable cone conjecture for projective primitive symplectic varieties with $\Q$-factorial terminal singularities, see the introduction. For this, we closely follow Markman's proof in the smooth case, see \cite[Section~6.5]{Mar11}. 

\begin{lemma}\label{lem:maxfaces}
Let $X$ be a primitive symplectic variety with $\Q$-factorial terminal singularities and assume that $b_2(X)\geq 5$.
Let $E$ be a prime exceptional divisor on $X$. Then $E^\perp\cap {\FEbar X}$ is a top dimensional cone in the hyperplane $E^\perp$. In particular, $W_X$ cannot be generated by any proper subset of reflections in prime exceptional divisors 
$\{R_e: e\ {\rm {prime\ exceptional}}\}$.
\end{lemma}
\begin{proof}
The proof is verbatim as in \cite[Proof of~Lemma 6.20]{Mar11}. Notice that the proof requires that 
$$
q_X(e,e')\geq 0
$$ 
for any two distinct prime divisors $e,e'$, a fact which is proved in \cite[Theorem~3.13]{KMPP19} for possibly singular symplectic varieties. 
\end{proof}

\begin{lemma}\label{lem:mov-ns}
Let $X$ be a projective primitive symplectic variety with $\Q$-factorial terminal singularities and assume that $b_2(X)\geq 5$.
Then
$$
 \Mov(X)^\circ= {\FE X}\cap (\Pic(X)\otimes \mathbb R).
$$
Moreover, $W_X$ acts faithfully on $\Pos(X)$ and each chamber of the positive cone intersects the algebraic positive cone, thus inducing a one-to-one correspondence between the set of exceptional chambers and the chambers in $\Pos(X)$ with respect to the $W_X$-action. In particular, the closure of movable cone $\Movbar(X)$ in $\Pos(X)$ is a fundamental domain for the action of $W_X$ on $\Pos(X)$.
\end{lemma}
\begin{proof}
One inclusion is given by Lemma \ref{lem:mov-pos}.
Conversely, if $L$ is $q_X$-nef on uniruled prime divisors on $X$, by the unicity of the  Boucksom-Zariski decomposition proved in  
\cite[Theorem~1.1]{KMPP19}, it coincides with its positive part, and therefore by (\ref{eq:star}), the bundle $L$ is movable. 

The rest of the statement is as in the smooth case, see \cite[Lemma~6.22]{Mar11}.
\end{proof}
Let 
$$
\rho: \Mth(X)\to O({\rm NS(X)})
$$
be the restriction homomorphism. Then the following holds.
\begin{lemma}\label{lem:semi-alg}\
Let $X$ be a projective primitive symplectic variety with $\Q$-factorial terminal singularities and assume that $b_2(X)\geq 5$.
\begin{enumerate}
\item The image $\Gamma$ of $\rho$ is a finite index subgroup of $\O^+({\rm NS(X)})$.
\item The kernel of $\rho$ is a subgroup of $\Mtb(X)$
\item $\Gamma$ is a semidirect product of its normal subgroup $W_X$ and the quotient group
$\Gamma_{\bir} :=\Mtb(X)/\ker(\rho) \subset \Gamma$.
\end{enumerate}
\end{lemma}
\begin{proof}
(1) The proof is exactly as the proof of \cite[Lemma 6.23, item (1)]{Mar11}. Note that the inclusion $\Mth(X)\subset \O^+(H^2(X,\Z)$ has finite index due to \cite[Theorem~1.2, item (1)]{BL18}. 

(2) The proof is exactly as the proof of \cite[Lemma 6.23, item (2)]{Mar11}. To have that the stabilizer of $\mathcal {FE}_{X}$ is $\Mtb(X)$ one needs to 
use Theorem \ref{thm:semidirect}, item (4).

(3) It follows immediately from item (2) and Theorem  \ref{thm:semidirect}, item (5).
\end{proof}

Let $\widebar \Mov^{\rm eff}(X)$ be the intersection of the closure of the movable cone $\Mov(X)$ with the convex cone
in $\Pic(X)\otimes \mathbb R$ generated by the classes of effective line bundles. Let $\widebar \Mov^+(X)$  be the convex hull of 
$\widebar \Mov(X)\cap (\Pic(X)\otimes \mathbb Q)$.
\begin{theorem}\label{thm:mov-conj}
Let $X$ be a projective primitive symplectic variety with $\Q$-factorial terminal singularities and assume that $b_2(X)\geq 5$. Then there exists a rational convex polyhedral cone $\Pi$ in $\widebar{\Mov}^+(X)$ such that $\Pi$
is a fundamental domain for the action of $\Gamma_{\bir}$  on $\widebar \Mov^+(X)$.
\end{theorem}
\begin{proof}[Proof of Theorem \ref{thm:mov-conj}]
The following assertions hold:
\begin{enumerate}
\item $\Pic(X)$ is a lattice of signature $(1,\bullet)$ and $\Gamma$ is an arithmetic subgroup of $\O^+(\Pic(X))$;
\item $W_X\subset \O^+(\Pic(X))$ is the reflection group generated by reflections in prime exceptional classes; 
\item  $\Gamma_{\bir}$ is equal to the stabilizer in  $\Gamma$ of prime exceptional classes;
\item $W_X$ is a normal subgroup of $\Gamma$ and the latter is semidirect product of $\Gamma_{\bir}$ and $W_X$. 
\item $\widebar \Mov(X)$ is a fundamental domain for the action of $W_X$ on $\mathcal \Pos(X)$ cut out by closed halfspaces
associated to prime exceptional classes.
\end{enumerate}
The claim about the signature in Item (1) is \cite[Lemma 5.3]{BL18} and \cite[Theorem~2, Item~(2)]{Sch20}, while the second part is the content of Lemma \ref{lem:semi-alg}, item (1).

Item (2) corresponds to Theorem \ref{thm:semidirect}, item (3).  

Item (3) follows from Theorem \ref{thm:semidirect}, item (4).

Item (4) is the content of Lemma \ref{lem:semi-alg}, item (3).

Item (5) is Lemma \ref{lem:mov-ns}.

As in the proof of the smooth case, see \cite[Theorem 6.25]{Mar11},  one follows Sterk's strategy \cite[Lemma 2.4]{sterk} to show, by using the above assertions, that for a class $x_0\in \Mov(X)$ not fixed by any $\gamma\in \Gamma\setminus\{{\id}\}$ the Dirichlet domain 
$$
\Pi := \{ x: q_X(x_0,x)\leq q_X(x_0, \gamma(x))\ \forall \gamma\in \Gamma\},
$$
 is rational and polyhedral and is the fundamental domain we are looking for.
\end{proof}
We deduce the following consequence.
\begin{theorem}\label{thm:finite}
For every integer $d\not=0$ the number of ${\Bir}(X)$-orbits of complete linear
systems, which contain an irreducible divisor of Beauville--Bogomolov--Fujiki degree $d$, is finite.
For every positive integer k there is only a finite number of ${\Bir}(X)$-orbits of complete
linear systems, which contain some irreducible divisor $L$  with $q_X(L)=0$ whose cohomology class is $k$ times a primitive isotropic class.
\end{theorem}
\begin{proof}
The proof is  as in the smooth case once one has the Boucksom-Zariski decomposition in the singular setting, provided by \cite[Theorem~1.1]{KMPP19},
and Theorem~\ref{thm:mov-conj}.
\end{proof}

\begin{proof}[Proof of Corollary \ref{cor:ogui}]
The proof goes  as in the smooth case treated in \cite[Proof of Theorem 1.3, (1) and (2)]{Ogui14}, once we replace \cite[Theorem 4.2]{Ogui14} with the birational cone conjecture proved above in the singular case and \cite[Proposition 4.1]{Ogui14} with (\ref{eq cones inclusions}) and Lemma \ref{lemma pos in big}. 
\end{proof}


\section{From the birational cone conjecture to the cone conjecture}\label{section birational to cone}

In this section, we will finish the proof of Conjecture \ref{conj:cone} in our setting.  In our strategy we will follow \cite{MY15}. 
The ingredients we need are the following:
\begin{enumerate}
\item The birational cone conjecture, which we proved in Theorem \ref{thm:mov-conj}.
\item The (locally trivial) deformation invariance of wall divisors, see Proposition~\ref{prop:walls_deform} below for a precise statement and a proof. 
\end{enumerate}

\begin{definition}\label{definition wall divisor}
Let $X$ be a projective primitive symplectic variety with $\Q$-factorial and terminal singularities and let $D$ be a Cartier divisor on $X$. Then $D$ is called a \emph{wall divisor} if the following  conditions hold:
\begin{enumerate}
    \item Negativity: $q_X(D)<0$.
    \item ``Wallness": for all $\varphi\in \Mth(X)$ we have $\varphi(D)^\perp \cap \BAmp(X)=\emptyset$.
    \end{enumerate}
\end{definition}

\begin{proposition}
Let $X$ be a primitive symplectic variety with $\Q$-factorial and terminal singularities and let $\pi:\, X\to Z $ be a relative Picard rank one contraction. Let $C$ be a contracted curve. Then the dual class $D$ to $C$ deforms to a wall divisor on all projective  deformations $X_t\in \Hdg_{[C]}$ in the Hodge locus of $C$.
\end{proposition}
\begin{proof}
By moving inside the Hodge locus of $C$ and deforming $\pi$ along the way as in the proof of Proposition \ref{prop defo}, we can directly assume that $X_t$ is projective, as projective deformations are dense in any family which is non trivial in moduli by \cite[Corollary~6.10]{BL18}. We consider a deformation $C_t\subset X_t$ of $C$ as in Corollary~\ref{corollary curve deforms} and denote $D_t=[C_t]^\vee$. We have to verify that $D_t$ is a wall divisor.

The negativity condition follows as $C$ is contracted (see \cite[Lemma~5.21]{BL18}) and $q_X$ is a topological invariant. For the wallness, let $\varphi\in\Mth(X)$ and consider a birational model $f:X_t\dashrightarrow X_t'$ such that $X_t'$ is $\Q$-factorial and terminal. Since $f_*$ is a Hodge isometry $(f_*\circ \varphi)\left([C_t]\right)$ remains of type $(2n-1,2n-1)$, i.e. it lies in the Hodge locus of $[C]$. As again by Corollary~\ref{corollary curve deforms} the curve $C_t$ deforms over all its Hodge locus, a (non-zero, possibly negative) multiple of $(f_*\circ \varphi)\left([C_t]\right)$ is effective and cannot be orthogonal to $\Amp(X_t')$.
As $D_t$ is dual to $C_t$, the same holds for $\varphi(D_t)$ and the claim is proven.
\end{proof}

\begin{lemma}\label{lem:wall_are_extremal}
Let $X$ be a projective primitive symplectic variety with $\Q$-factorial and terminal singularities with $b_2(X)\geq 5$ and let $D$ be a wall divisor. Then there exists $\varphi\in\Mth(X)$ and a birational model $X'$ such that $\varphi(D)$ is dual to an extremal curve $C'$ on $X'$ and $C'$ deforms in its Hodge locus.
\end{lemma}
\begin{proof}
Because $q_X(D)<0$, the hyperplane $D^\perp$ intersects the positive cone. By Theorem~\ref{thm:semidirect} (which requires the hypothesis $b_2(X)\geq 5$), there is a composition of reflections $\varphi$ such that $\varphi(D)^\perp$ intersects the closure of the fundamental exceptional chamber non-trivially. Therefore it follows that, for some $\varphi\in\Mth(X)$, $\varphi(D)^\perp$ does not intersect $\BAmp(X)$, but intersects $\widebar{\BAmp}(X)\cap \Pos(X)$. We deduce the existence of a birational model $X'$ of $X$ such that $\varphi(D)$ is orthogonal to a nef divisor $H'$. By applying the base-point-free theorem to $(X', H')$, we contract a curve $C'$ which is dual to $\varphi(D)$. The statement about deformations of $C'$ follows from Corollary~\ref{corollary curve deforms}.
\end{proof}

\begin{proposition}\label{prop:walls_deform}
Let $X$ be a projective primitive symplectic variety with $\Q$-factorial and terminal singularities and let $D$ be a wall divisor on $X$. Let $Y$ be another projective primitive symplectic variety such that there exists a parallel transport operator $\varphi$ from $X$ to $Y$ such that $\varphi(D)\in \Pic(Y)$. Then $\varphi(D)$ is a wall divisor.
\end{proposition}
\begin{proof}
Up to changing the birational model and acting with $\Mth(X)$, by Lemma \ref{lem:wall_are_extremal} we can assume without loss of generality that $D$ is dual (up to a sign) to an extremal curve $C$  which deforms along its Hodge locus. Therefore, $\varphi(D)$ is dual (up to a sign) to an effective curve on $Y$. Let $\psi$ be any element in $\Mth(X)$: we have that $\psi(\varphi(C))$ is still in the Hodge locus of $C$, hence also $\psi(\varphi(D))$ is dual to (up to sign) an effective curve, and our claim follows. 
\end{proof}

\begin{remark}\label{rmk:wall_versus_kahler}
In the non projective setting, ample classes are replaced by K\"ahler classes. By the above proposition, all parallel transport deformation of a wall divisor are dual to an effective curve (up to sign), hence they cannot be orthogonal to a K\"ahler class.
\end{remark}

The first step towards the conjecture is proving that the deformation invariance of wall divisors implies their boundedness, and this follows closely a proof of Amerik and Verbitsky \cite{AV20} in the smooth case, thanks to their proof of the following technical result:

\begin{theorem}\cite[Theorem 1.2]{AV20}\label{thm:orbitdense}
Let $(L, (\cdot,\cdot))$ be a lattice of signature (3,m) with $m\geq 2$. Let $V=L\otimes \R$ and let $\Gamma$ be an arithmetic subgroup of $\O(L)$. Let $R\subset L$ be a $\Gamma$-invariant set of primitive negative vectors of $L$. Let $Gr_{+}(3,V)$ be the Grassmannian of 3-spaces $W$ for which the restriction $(\cdot,\cdot)_{|W}$ is positive definite  and let $R^\perp$ be the set of three spaces orthogonal to an element of $R$. Then, either $\Gamma$ acts on $R$ with finitely many orbits or $R^\perp$ is dense. 
\end{theorem}

\begin{proposition}\label{prop:bounded_wall}
Let $X$ be a projective primitive symplectic variety with $\Q$-factorial and terminal singularities and $b_2(X)\geq 5$. Then there is a $B>0$, depending only on the locally trivial deformation class of $X$, such that 
$$q(D)\geq -B,$$
for all primitive wall divisors $D$ on $X$.
\end{proposition}
\begin{proof}
Let us suppose by contradiction that the square of primitive wall divisors is unbounded. This implies that there are infinitely many $\Mt(X)$-orbits of primitive wall divisors.
Set $L:= H^2(X,\Z)_\tf$. By  \cite[Theorem 8.2, item (1)]{BL18}, we may identify $\Mt(X)$ with a finite index subgroup of $\O(L)$. Notice that this requires $b_2(X)\geq 5$. 
Let us consider the Grassmannian $Gr_{+}(3,L\otimes\R)$ of positive 3 spaces inside $L\otimes \R$. By \cite[Proposition 5.15]{BL18}, to every K\"ahler class $\alpha$ on $X$ (or on a deformation of it), we can associate  a point inside $Gr_{+}(3,L\otimes\R)$ given by the 3-space $\Sigma_{X,\alpha}:=\langle \sigma_X+\widebar{\sigma_X},\sqrt{-1}(\widebar{\sigma_X}-\sigma_X),\alpha \rangle$, where $\sigma_X$ is the symplectic form on $X$. Let $\Omega\subset Gr_{+}(3,L\otimes\R)$ be the set of all $\Sigma_{Y,\beta}$ where $Y$ is a locally trivial deformation of $X$ and $\beta$ is a K\"ahler class on $Y$:
$$
\Omega:=\{\Sigma_{Y,\beta}\mid Y \sim_\lt X, \ \beta \textrm{ Kähler on }  Y\}.
$$
As for every deformation $Y$ of $X$ the K\"ahler cone is an open\footnote{This seems to be well-known, see e.g. \cite[Remark~2.6, item~(1)]{BL18}.} subset of $H^{1,1}(Y,\R)$, the set $\Omega$ contains an open  non-empty subset of $Gr_{+}(3,L\otimes\R)$. Let $D$ be a primitive wall divisor and let $M_D\subset Gr_{+}(3,L\otimes\R)$ be the set of three spaces orthogonal to an element in the $\Mt$-orbit of $D$. By Definition~\ref{definition wall divisor}, Proposition~\ref{prop:walls_deform} and Remark \ref{rmk:wall_versus_kahler}, we have 
\begin{equation}\label{eq:empty}
    \Omega \cap M_D=\emptyset.
\end{equation}
Let $M:=\cup_{D} M_D$, where the union runs over all primitive wall divisors on any projective locally trivial deformation of $X$. By applying Theorem \ref{thm:orbitdense} to $\Gamma=\Mt(X)$, $R$ the set of primitive wall divisors on some deformation of $X$, and $R^\perp=M$, we deduce that $M$ must be dense. However, by (\ref{eq:empty}), its complement contains $\Omega$, which in turn contains an open non-empty subset and we reach a contradiction. Therefore, there is a finite number of $\Mt(X)$-orbits of primitive wall divisors and their square is bounded by a rational constant.  
\end{proof}



\begin{proof}[Proof of Theorem \ref{thm:cone}, item (1)]
Let $\Pi$ be a rational polyhedral cone as in Theorem \ref{thm:mov-conj} for $\widebar{\Mov(X)}_{\mathbb Q}$. 
Let us consider the set 
$$\Sigma=\{ D\in \Pic(X)\,\mid D \text{ a primitive wall divisor}\}. $$
By Proposition \ref{prop:bounded_wall}, this set is contained inside
$$\Upsilon=\{ D\in\Pic(X):\ -B\leq D^2<0\}$$
for some $B>0$. Moreover, $\Sigma$ is invariant under the action of $\Mth(X)$ by definition.
By \cite[Proposition 3.4]{MY15}, elements in $\Upsilon$ whose orthogonal intersects $\Pi$ form a finite set, therefore the same holds for elements in $\Sigma$. Let $\Sigma_{\Pi}$ be this finite set:
\[
\Sigma_\Pi:=\{\alpha \in \Sigma \mid \alpha^\perp\cap \Pi \neq \emptyset\} 
\]
This gives us a finite subdivision 
\begin{equation}\label{eq:finitesub}
    \Pi = \cup_{i\in I} 
\Pi_i
\end{equation}
into closed rational polyhedral subcones of $\Pi$ cut out by elements in $\Sigma_\Pi$. These cones all have non-empty interior, as they are defined by a finite set of hyperplanes intersecting $\Pi$.
Let $\Pi_i$ be one such cone and let $g\in\Bir(X)$ be such that $g(\Pi_i)$ intersects $f^*(\Amp(Y))$ for some birational map $f:\,X\,\rightarrow\,Y$. 

We claim that 
\begin{equation}\label{eq:gPi}
g(\Pi_i)=f^*(\Nef(Y))\cap g(\Pi).
\end{equation}
Indeed, it suffices to prove the equality of the interiors: let us denote $\Pi^0$ and $\Pi_i^0$ the interiors of $\Pi$ and $\Pi_i$ respectively and let us consider the non-empty intersection 
$$\Lambda:= g(\Pi^0) \cap f^*(\Amp(Y)).$$ As both of these sets are convex, $\Lambda$ is connected. Furthermore, $\Lambda\subset g(\Pi_i)$ as it cannot intersect $D^\perp$ for any $D\in \Sigma_\Pi$. By Proposition~\ref{prop:fe-in-bdp}, the connected components of $\BAmp(X)\cap \Mov(X)=\BAmp(X)$ are all of the form $h^*(\Amp(Y'))$ for some $h\in\Bir(X)$, we also have $g(\Pi_i)\subset \Lambda$. 

Let us now fix a birational map $f:\,X\,\dashrightarrow\,Y$ and consider 
$$I_Y=\{i\in I,\text{such that }\exists\, g_i\in\Bir(X) \text{ with } g_i(\Pi_i)\subset f^*(\Nef(Y))\}.$$
This set is non-empty, as $\Pi$ is a fundamental domain for the $\Bir(X)$-action on $\widebar{\Mov}^+(X)$ and $f^*(\Nef(Y))$ is contained in it by definition.
Moreover, we claim that $I_Y$ is univoquely determined by the isomorphism class of $Y$, hence the notation is well defined.
Indeed, let $f$ and $f'$ be two birational maps from $X$ to $Y$, and let us suppose that $g_i(\Pi_i)\subset f^{*}(\Nef(Y))$ for a fixed $i$ and some $g_i\in\Bir(X)$. We need to prove that there exists $h_i\in\Bir(X)$ such that $h_i(\Pi_i)\subset f'^{*}(\Nef(Y))$. But this is easily verified with $h_i=(f'^{-1}\circ f)\circ g_i$.
Let us consider the set $I_X$ and let $i\in I_X$. The set 
$$\Bir_{i,X}:=\{g\in\Bir(X) \text{ such that } g(\Pi_i)\subset\Nef(X)\}
$$ verifies the following 
\begin{equation}\label{eq:aut-trans}
    \Aut(X)\cdot g = \Bir_{i,X},\ \forall g\in \Bir_{i,X}
\end{equation}
as for any two $g,h\in \Bir_{i,X}$ the pushforward $(g\circ h^{-1})_*$ sends an ample class to an ample class, hence $g\circ h^{-1}\in \Aut(X)$. 

Now notice that 
\begin{equation}
\Nef^+(X) \buildrel{(\ref{eq cones inclusions})}\over{\subset} \Movbar^+(X) \buildrel{Thm\  \ref{thm:mov-conj}}\over{\ \ =\ \ } \cup _{g\in \Bir (X)}\ g(\Pi)     
\buildrel{(\ref{eq:finitesub})}\over{=}
\cup _{g\in \Bir (X)}\ g(\Pi_i).
\end{equation}
Hence $\Nef^+(X)$ is equal to the union of the translates of the $\Pi_i$ intersecting its
interior, by (\ref{eq:gPi}), which are in turn unions of $\Aut(X)$-translates by (\ref{eq:aut-trans}). 

All in all we have checked that $\Nef^+(X)$ is a union of $\Aut(X)$-translates of the rational polyhedral cones $\Pi_i$, where $i$ is in the finite set $I_X$. 

Let $G$ be the image of $\Aut(X)$ inside $O(H^2(X,\mathbb{Z}))$ and let $y\in \Amp(X)\cap H^2(X,\mathbb{Q})$ be a rational ample class, such that its stabilizer in $G$ is trivial. Let us consider the following domain:
$$\mathcal{D}_y=\{  x\in\Nef(X), (x,y)\leq(x,g(y)),\,\,\forall\,g\in G\}. $$   
The conclusion now follows from a result hyperbolic geometry  which we recall for the reader's convenience.

\begin{lemma}[Lemma 2.8 in \cite{MY15}, Lemma 2.2 in \cite{totaro}, or  Application 4.14 in \cite{Loo2}]
Suppose we are given a finite set of rational polyhedral cones in $\Nef
(X)$,
such that $\Nef
(X)$ is the union of their $G$-translates. Let $y$ be a rational point with trivial stabilizer in $G$ lying in the
interior of one of these rational polyhedral cones. Then the Dirichlet domain $\mathcal D_y$ given above is rational polyhedral, it is contained in
$\Nef^+
(X)$, and $\Nef^+(X) = \cup_{g\in G}
g(\mathcal D_y)$.
\end{lemma}

$\mathcal{D}_y$ is a fundamental rational polyhedral domain for the action of $\Aut(X)$, and the theorem follows.
\end{proof}

In particular, we have obtained the following.

\begin{proof}[Proof of Corollary \ref{cor:finite}]
Following the proof of  Theorem \ref{thm:cone}, item (1), every birational model $Y$ of $X$ determines univoquely a set $I_Y\subset I$. As $I$ is finite, it has a finite number of partitions, hence our claim.
\end{proof}
\begin{proof}[Proof of Corollary
\ref{cor:finite-cont}]
One argues as in the Introduction of \cite{totaro}, by using Theorem~\ref{thm:cone}, item~(1).
\end{proof}

\bibliography{literatur}
\bibliographystyle{alpha}

\end{document}